\documentclass[11pt]{article}
\usepackage{amsmath}
\usepackage{amsfonts}
\usepackage{amssymb}
\usepackage{color}

\newtheorem{theorem}{Theorem}[section]

\newtheorem{corollary}[theorem]{Corollary}

\newtheorem{definition}[theorem]{Definition}
\newtheorem{example}[theorem]{Example}

\newtheorem{lemma}[theorem]{Lemma}

\newtheorem{proposition}[theorem]{Proposition}
\newtheorem{remark}[theorem]{Remark}

\newenvironment{proof}[1][Proof]{\noindent\textbf{#1.} }{\ \rule{0.5em}{0.5em}}
\input{tcilatex}
\begin{document}

\title{Strict pseudocontractions and demicontractions, \\
their properties and applications}
\author{Andrzej Cegielski \\
%EndAName
University of Zielona G\'{o}ra\\
Zielona G\'{o}ra, Poland\\
(a.cegielski@wmie.uz.zgora.pl)$\bigskip $}
\date{May, 2nd, 2023 }
\maketitle

\begin{abstract}
We give properties of strict pseudocontractions and demicontractions defined
on a Hilbert space, which constitute wide classes of operators that arise in
iterative methods for solving fixed point problems. In particular, we give
necessary and sufficient conditions under which a convex combination and
composition of strict pseudocontractions as well as demicontractions that
share a common fixed point is again a strict pseudocontraction or a
demicontraction, respectively. Moreover, we introduce a generalized
relaxation of composition of demicontraction and give its properties. We
apply these properties to prove the weak convergence of a class of
algorithms that is wider than the Douglas--Rachford algorithm and projected
Landweber algorithms. We have also presented two numerical examples, where
we compare the behavior of the presented methods with the Douglas-Rachford
method.

\medskip

\end{abstract}

\section{Introduction}

In the last three decades many researchers studied the properties of
averaged operators in real Hilbert spaces \cite{BBR78, BB96, Byr04, BC17,
Ceg12}. Here we recall that an averaged operator $T$ is defined as a strict
convex combination of a nonexpansive operator and identity. Important
examples of averaged operators in a real Hilbert space $\mathcal{H}$ are the
metric projection \cite{BB96}, the resolvent of a maximally monotone
operator \cite{BC17}, the proximity operator related to a lower
semi-continuous convex function \cite{BC17} and the Landweber operator \cite%
{Byr02, Byr04}. It is well known that convex combinations as well as
compositions of averaged operators are again averaged operators \cite{BB96,
OY02, Byr04, Ceg12}. This property enables construction of a wide class of
averaged operators which have many applications in signal processing, image
reconstruction from projections, medical imaging (in particular in
computerized tomography), radiation therapy treatment planning, optics,
supervised learning process and many other areas \cite{Ceg12, CZ97, Deu01,
SY98}. These problems may be modelled as common fixed point problems, split
feasibility problems or variational inequality problems. For these problems
appropriate methods employing convex combinations or compositions of
averaged operators have been constructed in many papers. In the last three
decades these methods were developed in many publications and extended to
algorithms employing averaged operators defined on a Hilbert space, see,
e.g. \cite{Byr04, BB96, BC17, Ceg12, Ceg15, CZ97} and the references
therein. Averaged operators $T$ with nonempty $\limfunc{Fix}T$ have a
convenient property, namely the \textit{demi closedness principle} 
\begin{equation*}
(x^{k}\rightharpoonup y\text{ and }\Vert T(x^{k})-x^{k}\Vert \rightarrow
0)\Longrightarrow y\in \limfunc{Fix}T
\end{equation*}%
\cite{Opi67}. This yields weak convergence of algorithms employing these
operators. However, evaluation of averaged operators with a specified subset
of fixed points is often not straightforward. Thus, classes of operators
wider than the averaged ones that have a fixed point, namely strongly
quasi-nonexpansive operators (also called strongly attracting), had to be
defined \cite{BR77, BB96, Ceg12}. Here we recall that an operator $T:%
\mathcal{H}\rightarrow \mathcal{H}$ with nonempty $\limfunc{Fix}T$ is called 
$\rho $-strongly quasi-nonexpansive, where $\rho \in (0,+\infty )$ if 
\begin{equation*}
\Vert T(x)-z\Vert ^{2}\leq \Vert x-z\Vert ^{2}-\rho \Vert T(x)-x\Vert ^{2}
\end{equation*}%
for all $x\in \mathcal{H}$ and all $z\in \limfunc{Fix}T$. In many cases
these operators may be evaluated in a much simpler manner than their
averaged counterparts \cite{Ceg12}. An important example of a strongly
quasi-nonexpansive operator in a Hilbert space is a subgradient projection
related to a lower semicontinuous convex function. Other examples include a
Landweber operator related to a strongly quasi-nonexpansive operator \cite%
{Ceg15, Ceg16, CGRZ20, WX11} and its extrapolation \cite{Ceg16, CGRZ20,
LMWX12}. It turns out that convex combinations as well as compositions of
strongly quasi-nonexpansive operators that share a common fixed point are
again strongly quasi-nonexpansive \cite{BB96, YO04, Ceg12}. Moreover, a
class of strongly quasi-nonexpansive operators that share a common fixed
point and satisfy the demi closedness principle is also closed under convex
combinations and compositions \cite{CRZ18}. This property is important for
the convergence properties of corresponding algorithms \cite{CRZ18}. In the
recent decade operators from wider classes than averaged operators and
quasi-nonexpansive operators, namely strict pseudocontractions and
demicontractions, have been applied for solving the problems mentioned
above, see, e.g. \cite{BDP22, Ber19, Ber23, CM10, Mai10, MX07, Mou10, SC16,
TPL12, WC13, Zho08}. Here we recall that an operator $T:\mathcal{H}%
\rightarrow \mathcal{H}$ with nonempty $\limfunc{Fix}T$ is called $\alpha $-%
\textit{strict pseudocontraction}, where $\alpha \in (-\infty ,1)$, if 
\begin{equation}
\Vert T(x)-T(y)\Vert ^{2}\leq \Vert x-y\Vert ^{2}+\alpha \Vert
(T(x)-x)-(T(y)-y)\Vert ^{2}  \label{e-SPC1}
\end{equation}%
for all $x,y\in \mathcal{H}$ \cite[Def. 1]{BP67}, \cite{HK77}. Properties of
strict pseudocontractions were studied recently in \cite{BDP22}, where these
operators were called \textit{conically averaged. }If we suppose that $%
\limfunc{Fix}T\neq \emptyset $ and set $y\in \limfunc{Fix}T$ in (\ref{e-SPC1}%
) then we obtain the definition of an $\alpha $-\textit{demicontraction} 
\cite{HK77}. If $\alpha \in (-\infty ,0)$ then $T$ is respectively: averaged
or $-\alpha $-strongly quasi-nonexpansive. Demicontractions were also used
in extrapolated versions of appropriate algorithms for the (multiple) split
common fixed point problem, see, e.g. \cite{CSHG19, Cui21, HS18, SCS20,
SSP19, YLP17}. Besides of a recent paper \cite{BDP22}, other papers
mentioned above do not study the properties of strict pseudocontractions and
demicontractions in detail. However, these properties are important for the
convergence of algorithms employing those operators, enable a simplification
of proofs of convergence and allow to apply parameters from a wider range.
The aim of this paper is answering the following questions:

\begin{enumerate}
\item Are classes of strict pseudocontractions and classes of
demicontractions that share a common fixed point and satisfy the demi
closedness principle closed under convex combinations?

\item Is it true (and under what conditions) that composition of strict
pseudocontractions as well as composition of demicontractions that share a
common fixed point and satisfy the demi closedness principle is again a
strict pseudocontraction or a demicontraction satisfying the demi closedness
principle, respectively?

\item What should we suppose on operators $T$ and $U$ being
pseudocontractions or demicontractions, on relaxation parameters $\lambda
_{k}$, $k\geq 0$, and on extrapolation function $\sigma :\mathcal{H}%
\rightarrow \lbrack 1,+\infty )$, employed in an iterative process $%
x^{k+1}=x^{k}+\lambda _{k}\sigma (x^{k})(UT(x^{k})-x^{k})$ in order to
guarantee weak convergence of $x^{k}$ to a fixed point of $UT$?
\end{enumerate}

Answers on the part of questions 1 and 2 regarding strict pseudocontractions
were recently presented in \cite[Props 2.4 and 2.5]{BDP22}. The main results
of this paper are contained in Subsections \ref{ss-3.2} and \ref{ss-3.3},
and in Section \ref{s-4}, where we give answers to the remaining parts of
questions 1 and 2 and to question 3.

\section{Preliminaries}

In the whole paper $\mathcal{H}$ denotes a real Hilbert space with inner
product $\langle \cdot ,\cdot \rangle $ and the related norm $\Vert \cdot
\Vert $. We suppose that $\mathcal{H}$ is nontrivial, that is $\mathcal{H}%
\neq \{0\}$. For an operator $T:\mathcal{H}\rightarrow \mathcal{H}$ and $%
\lambda >0$, denote by $T_{\lambda }:=\limfunc{Id}+\lambda (T-\limfunc{Id})$
the $\lambda $\textit{-relaxation} of $T$, where $\limfunc{Id}$ denotes the
identity operator. If $\lambda \in (0,1)$ then $T_{\lambda }$ is a convex
combination of $T$ and $\limfunc{Id}$. Note that $(T_{\lambda })_{\mu
}=T_{\lambda \mu }$, $\lambda ,\mu >0$. The set $\limfunc{Fix}T:=\{z\in 
\mathcal{H}:Tz=z\}$ is called the \textit{fixed point set} and an element of 
$\limfunc{Fix}T$ is called a \textit{fixed point}. Below, we recall some
well known notions which we use in this paper.

\begin{definition}
\label{d-NE}%
%TCIMACRO{\TeXButton{rm}{\rm\ }}%
%BeginExpansion
\rm\ %
%EndExpansion
We say that an operator $T:\mathcal{H}\rightarrow \mathcal{H}$ is:

\begin{enumerate}
\item[(a)] \textit{nonexpansive} (NE), if for all $x,y\in \mathcal{H}$%
\begin{equation}
\Vert T(x)-T(y)\Vert \leq \Vert x-y\Vert \text{;}  \label{e-NE}
\end{equation}

\item[(b)] $\alpha $\textit{-averaged (}$\alpha $-AV), where $\alpha \in
(0,1)$, if there is an NE operator $S$ such that 
\begin{equation*}
T=(1-\alpha )\limfunc{Id}+\alpha S\text{;}
\end{equation*}

\item[(c)] \textit{firmly nonexpansive} (FNE), if for all $x,y\in \mathcal{H}
$ 
\begin{equation}
\langle T(x)-T(y),x-y\rangle \geq \Vert T(x)-T(y)\Vert ^{2}\text{;}
\label{e-FNE}
\end{equation}

\item[(d)] $\lambda $-relaxed firmly nonexpansive ($\lambda $-RFNE), where $%
\lambda >0$, if $T$ is a $\lambda $-relaxation of an FNE operator;

\item[(e)] an $\alpha $\textit{-strict pseudocontraction (}$\alpha $-SPC)%
\textit{, where }$\alpha \in (-\infty ,1)$, if for all $x,y\in \mathcal{H}$ 
\begin{equation}
\Vert T(x)-T(y)\Vert ^{2}\leq \Vert x-y\Vert ^{2}+\alpha \Vert
(T(x)-x)-(T(y)-y)\Vert ^{2}\text{;}  \label{e-SPC}
\end{equation}
\end{enumerate}
\end{definition}

Averaged operators were introduced in \cite{BBR78}, where many properties of
these operators were presented. FNE operators were studied in \cite{BR77},
in \cite{Rei77} and in \cite[Section 11]{GR84}. Strict pseudocontractions
were studied in \cite[Def. 1]{BP67} and in \cite{HK77} In \cite[Def. 2.1]%
{BDP22}, a $\lambda $-RFNE operator was called $\Theta $-conically averaged,
where $\Theta =\lambda /2$.

\begin{proposition}
\label{p-RFNE} Let $\lambda >0$. An operator $T:\mathcal{H}\rightarrow 
\mathcal{H}$ is $\lambda $-RFNE if and only if for all $x,y\in \mathcal{H}$ 
\begin{equation}
\langle y-x,(T(x)-x)-(T(y)-y)\rangle \geq \frac{1}{\lambda }\Vert
(T(x)-x)-(T(y)-y)\Vert ^{2}\text{.}  \label{e-RFNE}
\end{equation}
\end{proposition}

\begin{proof}
See, \cite[Th. 1]{BP67} or \cite[Cor. 2.2.3]{Ceg12}.
\end{proof}

\begin{proposition}
\label{c-SPC-RFNE}Let $\lambda >0$. An operator $T:\mathcal{H}\rightarrow 
\mathcal{H}$ is $\lambda $-RFNE if and only if $T$ is an $\alpha $-SPC,
where $\alpha =(\lambda -2)/\lambda \in (-\infty ,1)$.
\end{proposition}

\begin{proof}
For $\lambda \in (0,2)$ the proposition was proved in \cite[Cor. 2.2.15]%
{Ceg12}. It follows from the proof of \cite[Cor. 2.2.15]{Ceg12}, that the
proposition is true for arbitrary $\lambda >0$. See also \cite[Prop. 2.2(iii)%
]{BDP22}.
\end{proof}

\bigskip

Clearly, an operator $T$ is NE if and only if $T$ is a $0$-SPC. Moreover,
Proposition \ref{c-SPC-RFNE} yields that an operator $T$ is $\frac{\lambda }{%
2}$-averaged (equivalently $\lambda $-RFNE with $\lambda \in (0,2)$) if and
only if $T$ is an $\alpha $-SPC, where $\alpha =1-\frac{2}{\lambda }<0$.

\begin{definition}
\label{d-QNE}%
%TCIMACRO{\TeXButton{rm}{\rm\ }}%
%BeginExpansion
\rm\ %
%EndExpansion
Let $T:\mathcal{H}\rightarrow \mathcal{H}$ be an operator with nonempty $%
\limfunc{Fix}T$. We say that $T$ is:

\begin{enumerate}
\item[(a)] \textit{quasi-nonexpansive} (QNE), if for all $x\in \mathcal{H}$
and $z\in \limfunc{Fix}T$%
\begin{equation}
\Vert T(x)-z\Vert \leq \Vert x-z\Vert \text{;}  \label{e-QNE}
\end{equation}

\item[(b)] $\rho $-\textit{strongly quasi-nonexpansive} ($\rho $-SQNE),
where $\rho \geq 0$, if for all $x\in \mathcal{H}$ and $z\in \limfunc{Fix}T$ 
\begin{equation}
\Vert T(x)-z\Vert ^{2}\leq \Vert x-z\Vert ^{2}-\rho \Vert T(x)-x\Vert ^{2}%
\text{;}  \label{e-SQNE}
\end{equation}%
If $\rho >0$ then we simply say that $T$ is SQNE;

\item[(c)] a \textit{cutter,} if for all $x\in \mathcal{H}$ and $z\in 
\limfunc{Fix}T$ 
\begin{equation}
\langle z-x,T(x)-x\rangle \geq \Vert T(x)-x\Vert ^{2}\text{;}  \label{e-cut}
\end{equation}

\item[(d)] a $\lambda $\textit{-relaxed cutter}, where $\lambda >0$, if $T$
is a $\lambda $-relaxation of a cutter, equivalently, for all $x\in \mathcal{%
H}$ and $z\in \limfunc{Fix}T$ 
\begin{equation}
\lambda \langle z-x,T(x)-x\rangle \geq \Vert T(x)-x\Vert ^{2}\text{;}
\label{e-RC}
\end{equation}

\item[(e)] an $\alpha $\textit{-demicontraction} (or an $\alpha $\textit{%
-demicontractive operator}), where $\alpha \in (-\infty ,1)$, if for all $%
x\in \mathcal{H}$ and $z\in \limfunc{Fix}T$ 
\begin{equation}
\Vert T(x)-z\Vert ^{2}\leq \Vert x-z\Vert ^{2}+\alpha \Vert T(x)-x\Vert ^{2}%
\text{.}  \label{e-DC}
\end{equation}
\end{enumerate}
\end{definition}

The name QNE operator was introduced by Dotson in \cite{Dot70} and the name
SQNE operator was introduced by Bruck in \cite{Bru82}. In \cite{Mar77} a
relaxed cutter was called an operator satisfying condition (A). An $\alpha $%
-strict pseudocontraction with a fixed point is an $\alpha $%
-demicontraction. Clearly, if $\alpha \leq 0$ then an $\alpha $%
-demicontraction is $\rho $-SQNE, where $\rho =-\alpha $, so the notion of a
demicontraction is an extension of the notion of an SQNE operator. An
example of FNE operator is the metric projection $P_{C}$ onto a closed
convex subset $C\subseteq \mathcal{H}$. In particular, for an affine
subspace $H\subseteq \mathcal{H}$ (e.g. a hyperplane), for all $x\in 
\mathcal{H}$ and $z\in H$ it holds 
\begin{equation*}
\langle z-x,P_{H}(x)-x\rangle =\Vert P_{H}(x)-x\Vert ^{2}\text{.}
\end{equation*}%
Thus, for any $\lambda >0$ and a $\lambda $-relaxed projection $%
(P_{H})_{\lambda }$ it holds 
\begin{equation*}
\langle z-x,(P_{H})_{\lambda }(x)-x\rangle =\frac{1}{\lambda }\Vert
(P_{H})_{\lambda }(x)-x\Vert ^{2}\text{,}
\end{equation*}%
$x\in \mathcal{H}$ and $z\in H$. Further properties of NE, AV, FNE, RFNE,
QNE, SQNE operators, cutters and relaxed cutters as well as relations among
these operators which we use in this paper can be found, e.g. in \cite[Secs
2.1 and 2.2]{Ceg12}. Below, we recall relations of demicontractions to
relaxed cutters and to SQNE operators. The following results are well known,
see, e.g. \cite[Rem. 2.3]{MaiM11}.

\begin{theorem}
\label{t-c-dc}Let $T:\mathcal{H}\rightarrow \mathcal{H}$ have a fixed point
and $\lambda >0$. The following conditions are equivalent:

\begin{enumerate}
\item[$\mathrm{(i)}$] $T$ is a cutter;

\item[$\mathrm{(ii)}$] $T_{\lambda }$ is an $\alpha $-demicontraction with $%
\alpha =(\lambda -2)/\lambda \in (-\infty ,1)$.
\end{enumerate}
\end{theorem}

\begin{proof}
If $\lambda \in (0,2]$ then it follows from \cite[Th. 2.1.39]{Ceg12} that $T$
is a cutter if and only if $T_{\lambda }$ is $\rho $-SQNE with $\rho
:=(2-\lambda )/\lambda \in \lbrack 0,+\infty )$. The second part of this
equivalence means that $T_{\lambda }$ is an $\alpha $-demicontraction with $%
\alpha =-\rho =(\lambda -2)/\lambda \in (-\infty ,0]$. It follows from the
proof of \cite[Th. 2.1.39]{Ceg12} that the theorem is true for arbitrary $%
\lambda >0$. If $\lambda >2$ then $\alpha =(\lambda -2)/\lambda \in (0,1)$.
\end{proof}

\bigskip

Because the function $f:(0,+\infty )\rightarrow (-\infty ,1)$ defined by $%
f(\lambda )=(\lambda -2)/\lambda $ is a bijection, Theorem \ref{t-c-dc}
states that every demicontraction can be treated as a relaxation of a cutter
and vice versa. Replacing $T_{\lambda }$ by $S$ in Theorem \ref{t-c-dc},
setting $\lambda :=2/(1-\alpha )$ for $\alpha \in (-\infty ,1)$ we obtain
the following result which is equivalent to Theorem \ref{t-c-dc}.

\begin{corollary}
\label{c-dc-c}Let $S:\mathcal{H}\rightarrow \mathcal{H}$ have a fixed point, 
$\alpha \in (-\infty ,1)$ and $\lambda =2/(1-\alpha )\in (0,+\infty )$. Then
the following conditions are equivalent:

\begin{enumerate}
\item[$\mathrm{(i)}$] $S$ is an $\alpha $-demicontraction;

\item[$\mathrm{(ii)}$] $S$ is a $\lambda $-relaxed cutter.
\end{enumerate}
\end{corollary}

Theorem \ref{t-c-dc} and Corollary \ref{c-dc-c} yields that a relaxation of
a demicontraction is again a demicontraction (cf. \cite[Lem. 3.1]{Ber23}).

\begin{corollary}
\label{c-dc-dc}Let $S:\mathcal{H}\rightarrow \mathcal{H}$ have a fixed
point, $\alpha \in (-\infty ,1)$ and $\mu >0$. Then the following conditions
are equivalent:

\begin{enumerate}
\item[$\mathrm{(i)}$] $S$ is an $\alpha $-demicontraction;

\item[$\mathrm{(ii)}$] $S_{\mu }$ is a $\beta $-demicontraction, where $%
\beta =(\mu +\alpha -1)/\mu \in (-\infty ,1)$.
\end{enumerate}
\end{corollary}

The notion of the demi closedness principle defined below is important for
the convergence properties of algorithms employing relaxed cutters.

\begin{definition}
%TCIMACRO{\TeXButton{rm}{\rm\ }}%
%BeginExpansion
\rm\ %
%EndExpansion
We say that an operator $T:\mathcal{H}\rightarrow \mathcal{H}$ satisfies the 
\textit{demi closedness principle} if $T-\limfunc{Id}$ is \textit{demiclosed}
at $0$, i.e. for any bounded sequence $\{x^{k}\}_{k=0}^{\infty }$ with $%
\Vert T(x^{k})-x^{k}\Vert \rightarrow 0$ and for its weak cluster point $y$
it holds that $y\in \limfunc{Fix}T$.
\end{definition}

In some publications the notion \textit{weak regularity} of $T$ is applied
for an operator $T$ satisfying the demi closedness principle, see, e.g. \cite%
{CRZ18}. It is well known that a nonexpansive operator satisfies the demi
closedness principle \cite{Opi67}. Obviously, a relaxation of an operator
satisfying the demi closedness principle also satisfies the demi closedness
principle. In particular, a strict pseudocontraction, as a relaxation of an
FNE operator, satisfies the demi closedness principle. Moreover, a convex
combination as well as composition of SQNE operators that share a common
fixed point and satisfy the demi closedness principle also satisfies the
demi closedness principle \cite[Thms 4.1 and 4.2]{Ceg15}. In what follows,
we give conditions under which strict pseudocontractions as well as
demicontractions also share these properties. It is well known that fixed
point iterations employing RFNE operators as well as algorithms employing
relaxed cutters (or, equivalently, SQNE operators) that satisfy the demi
closedness principle generate sequences which converge weakly, under the
assumption that the relaxation parameters are in $(0,2)$. We briefly recall
corresponding results. Let $V:\mathcal{H}\rightarrow \mathcal{H}$ be an
operator with a fixed point. Consider the following iteration 
\begin{equation}
x^{k+1}=V_{\nu _{k}}(x^{k})\text{,}  \label{e-iterT-lamk}
\end{equation}%
where $x^{0}\in \mathcal{H}$ is arbitrary and $\nu _{k}\geq 0$ is a
relaxation parameter applied in the $k$-th iteration, $k\geq 0$. The
following result which is due to Reich \cite[Theorem 2]{Rei79} is also known
in the literature as the Krasnosel'ski\u{\i}-Mann theorem (see, e.g. \cite[%
Thms 2.1 and 2.2]{Byr04}).

\begin{proposition}
\label{p-KM}If $V$ is FNE with $\limfunc{Fix}V\neq \emptyset $, $\nu _{k}\in
\lbrack \varepsilon ,2-\varepsilon ]$ for some $\varepsilon \in (0,1)$ and $%
x^{k}$ is generated by iteration $\mathrm{(\ref{e-iterT-lamk})}$ then $x^{k}$
converges weakly to an element of $\limfunc{Fix}V$.
\end{proposition}

Many variants of the following result are well known (see, e.g. \cite[Cor.
3.7.3]{Ceg12} or its more general versions \cite[Th. 2.9]{BC01} or \cite[Th.
6.1(i)]{CRZ18}).

\begin{proposition}
\label{p-WC}If $V$ is a cutter satisfying the demi closedness principle, $%
\nu _{k}\in \lbrack \varepsilon ,2-\varepsilon ]$ for some $\varepsilon \in
(0,1)$ and $x^{k}$ is generated by iteration $\mathrm{(\ref{e-iterT-lamk})}$
then $x^{k}$ converges weakly to an element of $\limfunc{Fix}V$.
\end{proposition}

An equivalent formulation of the proposition below was proposed by M\u{a}ru%
\c{s}ter in \cite[Th. 1]{Mar77}. By the first look this result seems to be
more general than Proposition \ref{p-WC} (see also \cite[Th. 12]{BP67} for a
related result).

\begin{proposition}
\label{t-Mar77}Let $\alpha \in (-\infty ,1)$, $V:\mathcal{H}\rightarrow 
\mathcal{H}$ be an $\alpha $-demicontraction satisfying the demi closedness
principle and the sequence $\{x^{k}\}_{k=0}^{\infty }$ be generated by
iteration $\mathrm{(\ref{e-iterT-lamk})}$. If $\nu _{k}\in \lbrack
\varepsilon ,1-\alpha -\varepsilon ]$ for some $\varepsilon \in (0,\frac{%
1-\alpha }{2})$, then $x^{k}$ converges weakly to an element of $\limfunc{Fix%
}V$.
\end{proposition}

\begin{remark}
%TCIMACRO{\TeXButton{rm}{\rm\ }}%
%BeginExpansion
\rm\ %
%EndExpansion
If $\alpha =-1$ in Proposition \ref{t-Mar77} then, by Corollary \ref{c-dc-c}%
, $V$ is a cutter ($1$-RFNE operator). Thus, Proposition \ref{p-WC} follows
from Proposition \ref{t-Mar77}. Nevertheless, Proposition \ref{t-Mar77} can
also be reduced to Proposition \ref{p-WC}. Indeed. Suppose that $V$ is an $%
\alpha $-demicontraction and $\nu _{k}\in \lbrack \varepsilon ,1-\alpha
-\varepsilon ]$ for some $\varepsilon \in (0,\frac{1-\alpha }{2})$. By
Corollary \ref{c-dc-c}, $T:=V_{\frac{1-\alpha }{2}}$ is a cutter.
Consequently, $V_{\nu _{k}}=T_{\lambda _{k}}$, where $\lambda _{k}=\frac{%
2\nu _{k}}{1-\alpha }\in \lbrack \frac{2\varepsilon }{1-\alpha },2-\frac{%
2\varepsilon }{1-\alpha }]$, i.e. $V$ satisfies the assumptions of
Proposition \ref{p-WC}. We see that the assumption that $V$ is an $\alpha $%
-demicontraction in Proposition \ref{t-Mar77} is superfluous. It is enough
to suppose that $V$ is a cutter and $\nu _{k}\in \lbrack \varepsilon
,2-\varepsilon ]$ for some $\varepsilon \in (0,1)$ or, equivalently, that $V$
is quasi-nonexpansive and $\nu _{k}\in \lbrack \varepsilon ,1-\varepsilon ]$
for some $\varepsilon \in (0,1/2)$.
\end{remark}

Many iterations studied in the last decades can be presented as special
cases of (\ref{e-iterT-lamk}). We recall here two algorithms.

\begin{example}
\label{ex-DR}%
%TCIMACRO{\TeXButton{rm}{\rm\ }}%
%BeginExpansion
\rm\ %
%EndExpansion
Let $T$ and $U$ be FNE. Define $V:=U_{2}T_{2}$ ($T_{2}$ and $U_{2}$ are NE
as $2$-relaxations of FNE operators) and let $\nu _{k}=\nu =\frac{1}{2}$.
Then $U_{2}T_{2}$ is NE as composition of NE operators, consequently $V_{\nu
}$ is FNE and (\ref{e-iterT-lamk}) with $\nu _{k}=\frac{1}{2}$ is, actually,
the \textit{averaged alternating reflection method} (also called \textit{%
Douglas--Rachford method}), see \cite{BC17}. Clearly, for arbitrary $\nu \in
(0,1)$ the operator $V_{\nu }$ is averaged. If, additionally, $\limfunc{Fix}%
V\neq \emptyset $ then $V_{\nu }$ is an SQNE operator satisfying the demi
closedness principle, thus, for any sequence $\{x^{k}\}_{k=0}^{\infty }$
generated by the iteration $x^{k+1}=V_{\nu _{k}}x^{k}$ where $\nu _{k}\in
\lbrack \varepsilon ,1-\varepsilon ]$ for some $\varepsilon \in (0,1/2)$,
converges weakly to a fixed point of $V$. Now define $V:=U_{\mu }T_{\lambda
} $ where $\lambda ,\mu >0$ (or, equivalently, $V$ is composition of an $%
\alpha $- and a $\beta $-SPC, where $\alpha ,\beta \in (-\infty ,1)$). Under
what conditions on $\lambda ,\mu >0$ (equivalently on $\alpha ,\beta \in
(-\infty ,1)$) and $\nu _{k}$ any sequence $\{x^{k}\}_{k=0}^{\infty }$
generated by iteration (\ref{e-iterT-lamk}) converges weakly to a fixed
point of $V$? In particular, may one of the parameters $\lambda $ or $\mu $
be greater than $2$? For example, if $V=U_{\mu }T_{\lambda }$, where $%
\lambda +\mu =4$, does the convergence hold for some $\nu _{k}$? Answer on
this questions follows from \cite[Props 2.5 and 2.9]{BDP22}. Natural
questions arise at this point. Does the convergence remain true if we
suppose that $T$ and $U$ are cutter operators that share a common fixed
point and satisfy the demi closedness principle instead of the assumption
that $T$ and $U$ are FNE and $\limfunc{Fix}V\neq \emptyset $? Under what
conditions on $\lambda $ and $\mu $ it holds $\limfunc{Fix}V=\limfunc{Fix}%
T\cap \limfunc{Fix}U$? These questions will be answered in Sections \ref{s-3}
and \ref{s-4}.
\end{example}

\begin{example}
\label{ex-CQ}%
%TCIMACRO{\TeXButton{rm}{\rm\ }}%
%BeginExpansion
\rm\ %
%EndExpansion
In \cite{Mou10} Moudafi considered the following \textit{split common fixed
point problem} (SCFPP) 
\begin{equation}
\text{find }x^{\ast }\in \limfunc{Fix}U\text{ such that }Ax^{\ast }\in 
\limfunc{Fix}T\text{,}  \label{e-SCFP}
\end{equation}%
where $A:\mathcal{H}_{1}\rightarrow \mathcal{H}_{2}$ is a nonzero bounded
linear operator, $\mathcal{H}_{1}\ $and $\mathcal{H}_{2}$ are two real
Hilbert spaces, $S:\mathcal{H}_{2}\rightarrow \mathcal{H}_{2}$ is an $\alpha 
$-demicontraction with $\limfunc{Fix}S=Q$, $U:$ $\mathcal{H}_{1}\rightarrow 
\mathcal{H}_{1}$ is a $\beta $-demicontraction with $\limfunc{Fix}U=C$, both
satisfying the demi closedness principle, and $F:=C\cap A^{-1}(Q)\neq
\emptyset $. Moudafi proposed the following iteration for solving the SCFPP%
\begin{equation}
x^{k+1}=U_{\mu _{k}}T_{\lambda }(x^{k})\text{,}  \label{e-iter-Mou10}
\end{equation}%
where $x^{0}\in \mathcal{H}_{1}$ is arbitrary, $\mu _{k}\in (\varepsilon
,1-\beta -\varepsilon )$ for some $\varepsilon \in (0,\frac{1-\beta }{2})$, $%
\lambda \in (0,1-\alpha )$ and $T$ is the Landweber operator related to $S$,
i.e.%
\begin{equation}
T(x):=\mathcal{L}\{S\}(x)=x+\frac{1}{\Vert A\Vert ^{2}}A^{\ast }(S(Ax)-Ax)
\label{e-LT}
\end{equation}%
(see \cite{CRZ20} for a definition of the Landweber transform $\mathcal{L}%
\{\cdot \}$). Moudafi proved that for arbitrary $x^{0}\in \mathcal{H}_{1}$
the sequence $\{x^{k}\}_{k=0}^{\infty }$ generated by iteration (\ref%
{e-iter-Mou10}), where $T$ is given by (\ref{e-LT}), converges weakly to an
element of $F$ \cite[Th. 2.1]{Mou10}. Related algorithms for the so called 
\textit{multiple split fixed point problems }(MSFPP) with cyclic application
of Landweber transform were studied in \cite{TPL12} and \cite{WC13}. Moudafi
supposed that $\alpha ,\beta \in \lbrack 0,1)$, but the convergence also
holds for arbitrary $\alpha ,\beta \in (-\infty ,1)$. Indeed. Suppose for
simplicity, that $\mu _{k}$ is constant, i.e. $\mu _{k}=\mu \in (0,1-\beta )$
for all $k\geq 0$. Similarly as in \cite[Lem. 4.1]{Ceg15} one can prove,
that $T$ is an $\alpha $-demicontraction and $\limfunc{Fix}T=A^{-1}(\limfunc{%
Fix}S)$. By Corollary \ref{c-dc-dc}, $T_{\lambda }$ is a $\gamma $%
-demicontraction, where $\gamma =1-\frac{1-\alpha }{\lambda }<0$, i.e. $%
T_{\lambda }$ is $-\gamma $-SQNE, and $U_{\mu }$ is a $\delta $%
-demicontraction, where $\delta =1-\frac{1-\beta }{\mu }<0$, i.e. $U_{\mu }$
is $-\delta $-SQNE. Thus $U_{\mu }T_{\lambda }$ is $\kappa $-SQNE, where $%
\kappa =-(\gamma ^{-1}+\delta ^{-1})^{-1}$ \cite[Cor. 2.1.47]{Ceg12}, $%
\limfunc{Fix}(U_{\mu }T_{\lambda })=\limfunc{Fix}U_{\mu }\cap \limfunc{Fix}%
T_{\lambda }=\limfunc{Fix}U\cap \limfunc{Fix}T=F$ \cite[Th. 2.1.26(ii)]%
{Ceg12} and $U_{\mu }T_{\lambda }$ satisfies the demi closedness principle 
\cite[Cor. 5.6(i)]{CRZ18}. Proposition \ref{p-WC} yields now that any
sequence $\{x^{k}\}_{k=0}^{\infty }$ generated by iteration (\ref%
{e-iter-Mou10}), where $T$ is given by (\ref{e-LT}), converges weakly to an
element of $F$. The convergence also holds if $\mu _{k}$ is not constant. To
prove it, one should introduce a definition of a sequence of operators
satisfying the demi closedness principle and apply \cite[Cor. 5.5(i)]{CRZ18}
instead of \cite[Cor. 5.6(i)]{CRZ18}. Note that iteration (\ref{e-iter-Mou10}%
) employed, actually, SQNE operators $U_{\mu _{k}}$ and $T_{\lambda }$.
Natural questions arise at this point. Is it possible to allow that $%
T_{\lambda }$ and/or $U_{\mu _{k}}$ are demicontractions which are not SQNE?
Under what conditions on $\lambda ,\mu ,\nu _{k}>0$ any sequence $%
\{x^{k}\}_{k=0}^{\infty }$ generated by the iteration $x^{k+1}=(U_{\mu
}T_{\lambda })_{\nu _{k}}(x^{k})$ converges weakly to an element of $F$?
Theses questions will be answered in Sections \ref{s-3} and \ref{s-4}.
\end{example}

\section{Properties of strict pseudocontractions and demicontractions \label%
{s-3}}

In this section we give conditions under which convex combinations as well
as compositions of demicontractions (satisfying the demi closedness
principle) are again demicontractions (satisfying the demi closedness
principle).

\subsection{Convex combinations of strict pseudocontractions and
demicontractions \label{ss3.1}}

Let $T_{i}:\mathcal{H}\rightarrow \mathcal{H}$, $i\in I:=\{1,2,...,m\}$. The
following theorems extend \cite[Thms 2.1.50 and 2.2.35]{Ceg12}, where only
the case $\lambda _{i}\in (0,2)$, $i\in I$, was considered. An equivalent
formulation of the theorem below is presented in \cite[Prop. 2.4]{BDP22}.

\begin{theorem}
\label{t-conv-RFNE}Let $T_{i}$ be $\lambda _{i}$-RFNE where $\lambda _{i}\in
(0,+\infty )$, $i\in I$, $w_{i}>0,i\in I$, with $\sum_{i=1}^{m}w_{i}=1$ and $%
T:=\sum_{i=1}^{m}w_{i}T_{i}$. Then:

\begin{enumerate}
\item[$\mathrm{(i)}$] $T$ is $\lambda $-RFNE with 
\begin{equation}
\lambda =\sum_{i=1}^{m}w_{i}\lambda _{i}\text{;}  \label{e-cc-lam}
\end{equation}

\item[$\mathrm{(ii)}$] The relaxation parameter $\lambda $ satisfies
inequalities 
\begin{equation}
0<\min_{i\in I}\lambda _{i}\leq \lambda \leq \max_{i\in I}\lambda
_{i}<+\infty \text{.}  \label{e-cc-lam-mm}
\end{equation}
\end{enumerate}
\end{theorem}

\begin{proof}
In \cite[Th. 2.2.35]{Ceg12} the case $\lambda _{i}\in (0,2)$, $i\in I$, was
considered. It follow from the proof of \cite[Th. 2.2.35]{Ceg12} that the
theorem is true for arbitrary $\lambda _{i}\in (0,+\infty )$, $i\in I$.
\end{proof}

\begin{corollary}
\label{c-conv-SPC}Let $T_{i}$ be an $\alpha _{i}$-SPC, where $\alpha _{i}\in
(-\infty ,1)$, $i\in I$, $w_{i}>0,i\in I$, with $\sum_{i=1}^{m}w_{i}=1$, and 
$T:=\sum_{i=1}^{m}w_{i}T_{i}$. Then:

\begin{enumerate}
\item[$\mathrm{(i)}$] $T$ is an $\alpha $-SPC with 
\begin{equation}
\alpha =1-\left( \sum_{i=1}^{m}\frac{w_{i}}{1-\alpha _{i}}\right) ^{-1}\text{%
;}  \label{e-cc-al}
\end{equation}

\item[$\mathrm{(ii)}$] The parameter $\alpha $ satisfies inequalities%
\begin{equation}
-\infty <\min_{i\in I}\alpha _{i}\leq \alpha \leq \max_{i\in I}\alpha _{i}<1%
\text{.}  \label{e-cc-al-mm}
\end{equation}
\end{enumerate}
\end{corollary}

\begin{proof}
By Proposition \ref{c-SPC-RFNE}, $T_{i}$ is $\lambda _{i}$-RFNE, where $%
\lambda _{i}=\frac{2}{1-\alpha _{i}}$, $i\in I$. Theorem \ref{t-conv-RFNE}
yields that $T$ is $\lambda $-RFNE, where $\lambda =\sum_{i=1}^{m}\frac{%
2w_{i}}{1-\alpha _{i}}$. Applying Proposition \ref{c-SPC-RFNE} again, we
obtain that $T$ is an $\alpha $-SPC with 
\begin{equation*}
\alpha =1-\frac{2}{\lambda }=1-2\left( \sum_{i=1}^{m}\frac{2w_{i}}{1-\alpha
_{i}}\right) ^{-1}=1-\left( \sum_{i=1}^{m}\frac{w_{i}}{1-\alpha _{i}}\right)
^{-1}\text{.}
\end{equation*}%
This proves part (i). Moreover 
\begin{eqnarray*}
-\infty  &<&\min_{j\in I}\alpha _{j}=1-\left( \sum_{i=1}^{m}\frac{w_{i}}{%
1-\min_{j\in I}\alpha _{j}}\right) ^{-1}\leq 1-\left( \sum_{i=1}^{m}\frac{%
w_{i}}{1-\alpha _{i}}\right) ^{-1} \\
&\leq &1-\left( \sum_{i=1}^{m}\frac{w_{i}}{1-\max_{j\in I}\alpha _{j}}%
\right) ^{-1}=\max_{j\in I}\alpha _{j}<1
\end{eqnarray*}%
which proves part (ii).
\end{proof}

\begin{theorem}
\label{t-conv-dc}Let $T_{i}:\mathcal{H}\rightarrow \mathcal{H}$ be a $%
\lambda _{i}$-relaxed cutter, where $\lambda _{i}\in (0,+\infty )$, $i\in I$%
, $\bigcap_{i=1}^{m}\limfunc{Fix}T_{i}\neq \emptyset $, $w_{i}>0,i\in I$,
with $\sum_{i=1}^{m}w_{i}=1$, and $T=\sum_{i=1}^{m}w_{i}T_{i}$. Then:

\begin{enumerate}
\item[$\mathrm{(i)}$] The operator $T$ is a $\lambda $-relaxed cutter with $%
\lambda $ given by $\mathrm{(\ref{e-cc-lam})}$;

\item[$\mathrm{(ii)}$] The relaxation parameter $\lambda $ satisfies
inequalities $\mathrm{(\ref{e-cc-lam-mm})}$;

\item[$\mathrm{(iii)}$] $\limfunc{Fix}T=\bigcap_{i=1}^{m}\limfunc{Fix}T_{i}$;

\item[$\mathrm{(iv)}$] Suppose that $T_{i}$, $i\in I$, satisfy the demi
closedness principle. Then $T$ also satisfies the demi closedness principle.
\end{enumerate}
\end{theorem}

\begin{proof}
(cf. \cite[Th. 2.1.50]{Ceg12}) Let $x\in \mathcal{H}$ and all $z\in
\bigcap_{i\in I}\limfunc{Fix}T_{i}$. Denote $v_{i}:=\frac{w_{i}\lambda _{i}}{%
\lambda }$ and define $U_{i}:=(T_{i})_{\lambda _{i}^{-1}}$ , $i\in I$. Then $%
U_{i}$ is a cutter, $T_{i}x-x=\lambda _{i}(U_{i}-x)$, $v_{i}>0$, $i\in I$,
and $\sum_{i=1}^{m}v_{i}=1$. By the convexity of the function $\Vert \cdot
\Vert ^{2}$ we obtain 
\begin{eqnarray}
\langle z-x,T(x)-x\rangle &=&\sum_{i=1}^{m}w_{i}\langle z-x,T_{i}(x)-x\rangle
\label{e-cc1} \\
&\geq &\sum_{i=1}^{m}\frac{w_{i}}{\lambda _{i}}\Vert T_{i}(x)-x\Vert
^{2}=\lambda \sum_{i=1}^{m}v_{i}\Vert U_{i}(x)-x\Vert ^{2}  \label{e-cc2} \\
&\geq &\lambda \left\Vert \sum_{i=1}^{m}v_{i}(U_{i}(x)-x)\right\Vert ^{2}=%
\frac{1}{\lambda }\left\Vert \sum_{i=1}^{m}w_{i}\lambda
_{i}(U_{i}(x)-x)\right\Vert ^{2}  \label{e-cc3} \\
&=&\frac{1}{\lambda }\left\Vert \sum_{i=1}^{m}w_{i}(T_{i}(x)-x)\right\Vert
^{2}=\frac{1}{\lambda }\Vert T(x)-x\Vert ^{2}\text{.}  \label{e-cc4}
\end{eqnarray}

(ii) is obvious.

(i) and (iii) The inclusion $\bigcap_{i=1}^{m}\limfunc{Fix}T_{i}\subseteq 
\limfunc{Fix}T$ is clear. If $\limfunc{Fix}T=\emptyset $ then the converse
inclusion is obvious. Suppose that $\limfunc{Fix}T\neq \emptyset $. For $%
x\in \limfunc{Fix}T$ there hold equalities in (\ref{e-cc1})-(\ref{e-cc4})
which in view of $\frac{w_{i}}{\lambda _{i}}>0$, $i\in I$, means that $x\in 
\limfunc{Fix}T_{i}$, $i\in I$. This yields (i) and (iii).

(iv) Define $U:=\sum_{i=1}^{m}v_{i}U_{i}$. By (i), $U$ is a cutter. We have 
\begin{equation*}
Tx-x=\sum_{i=1}^{m}w_{i}(T_{i}x-x)=\lambda \sum_{i=1}^{m}\frac{w_{i}\lambda
_{i}}{\lambda }(U_{i}x-x)=\lambda \sum_{i=1}^{m}v_{i}(U_{i}x-x)=\lambda
(Ux-x)\text{.}
\end{equation*}%
Because $U_{i}$ satisfies the demi closedness principle as a relaxation of $%
T_{i}$, $i\in I$, the operator $U$ satisfies the demi closedness principle 
\cite[Th. 4.1]{Ceg15}. Consequently, $T$ also satisfies the demi closedness
principle as a relaxation of $U$.
\end{proof}

\begin{corollary}
\label{c-conv-dc}Let $T_{i}:\mathcal{H}\rightarrow \mathcal{H}$ be an $%
\alpha _{i}$-demicontraction, where $\alpha _{i}\in (-\infty ,1)$, $i\in I$, 
$\bigcap_{i=1}^{m}\limfunc{Fix}T_{i}\neq \emptyset $, $w_{i}>0,i\in I$, with 
$\sum_{i=1}^{m}w_{i}=1$, and $T:=\sum_{i=1}^{m}w_{i}T_{i}$. Then:

\begin{enumerate}
\item[$\mathrm{(i)}$] $T$ is an $\alpha $-demicontraction with $\alpha $
given by $\mathrm{(\ref{e-cc-al})}$;

\item[$\mathrm{(ii)}$] The parameter $\alpha $ satisfies inequalities $%
\mathrm{(\ref{e-cc-al-mm})}$;

\item[$\mathrm{(iii)}$] $\limfunc{Fix}T=\bigcap_{i=1}^{m}\limfunc{Fix}T_{i}$;

\item[$\mathrm{(iv)}$] Suppose that $T_{i}$, $i\in I$, satisfy the demi
closedness principle. Then $T$ also satisfies the demi closedness principle.
\end{enumerate}
\end{corollary}

\subsection{Composition of strict pseudocontractions and demicontractions 
\label{ss-3.2}}

Before we formulate the main result of this paper, we prove an auxiliary
lemma. Let $\lambda ,\mu >0$. Consider the following equation 
\begin{equation}
\left( 1-\frac{2}{\nu }\right) ^{2}=4(\frac{1}{\lambda }-\frac{1}{\nu })(%
\frac{1}{\mu }-\frac{1}{\nu })\text{.}  \label{e-gam-eq}
\end{equation}%
If $\lambda =2$ and $\mu =2$ then arbitrary $\nu \neq 0$ is a solution of (%
\ref{e-gam-eq}). Otherwise, if $\lambda \mu =4$ then (\ref{e-gam-eq}) has no
solution. One can easily check that in other cases the unique solution of
equation (\ref{e-gam-eq}) is 
\begin{equation}
\nu ^{\ast }=\nu (\lambda ,\mu )=\frac{4(\lambda +\mu -\lambda \mu )}{%
4-\lambda \mu }\text{.}  \label{e-gam-sol}
\end{equation}

In particular, if $\lambda =2$ and $\mu \neq 2$ or if $\mu =2$ and $\lambda
\neq 2$ then $\nu ^{\ast }=2$ is the unique solution of (\ref{e-gam-eq}).
Level lines of the function $\nu (\lambda ,\mu )$ are presented on Figure 1.
Note that $\nu (\lambda ,\mu )$ is well defined if and only if $\lambda \mu
\neq 4$. Moreover, one can be easily proved that if $\lambda \mu <4$ then $%
\lambda +\mu >\lambda \mu $ and $\nu (\lambda ,\mu )>0$ (see Appendix). If $%
4<\lambda \mu <\lambda +\mu $ then $\nu (\lambda ,\mu )<0$. If $\lambda +\mu
<\lambda \mu $ then $\lambda \mu >4$ and $\nu (\lambda ,\mu )>0$. However,
in the latter case $\nu \leq \min \{\lambda ,\mu \}$, where the equality
holds if and only if $\min \{\lambda ,\mu \}=2$.\vspace{-0.5cm}

\begin{equation*}
\FRAME{itbpFU}{3.4694in}{3.3439in}{0in}{\Qcb{Figure 1. Solution $\protect\nu 
$ of equation (\protect\ref{e-gam-eq}) depending on $\protect\lambda $ and $%
\protect\mu $}}{}{Figure}{\special{language "Scientific Word";type
"GRAPHIC";maintain-aspect-ratio TRUE;display "USEDEF";valid_file "T";width
3.4694in;height 3.3439in;depth 0in;original-width 12.1839in;original-height
11.7375in;cropleft "0";croptop "1";cropright "1";cropbottom "0";tempfilename
'RXQPS701.wmf';tempfile-properties "XPR";}}
\end{equation*}%
\vspace{-0.5cm}

\bigskip

The following auxiliary lemma shows the properties of the solution of (\ref%
{e-gam-eq}) for $\lambda \mu <4$ (cf. \cite[Lem. 2.1.45]{Ceg12}, where the
case $\lambda ,\mu <2$ was considered).

\begin{lemma}
\label{l-ni}Let $\lambda ,\mu >0$ be such that $\lambda \mu <4$. Then:

\begin{enumerate}
\item[$\mathrm{(i)}$] The unique solution $\nu ^{\ast }=\nu (\lambda ,\mu )$
of equation $\mathrm{(\ref{e-gam-eq})}$ satisfies the following inequalities%
\begin{equation}
0<\min \{\lambda ,\mu \}<\frac{4\min \{\lambda ,\mu \}}{\min \{\lambda ,\mu
\}+2}\leq \nu ^{\ast }  \label{e-gam-ineq1}
\end{equation}%
and 
\begin{equation}
\nu ^{\ast }\geq \max \{\lambda ,\mu \}\text{;}  \label{e-gam-ineq2}
\end{equation}

\item[$\mathrm{(ii)}$] If, additionally, $\lambda ,\mu <2$ then%
\begin{equation}
\nu ^{\ast }\leq \frac{4\max \{\lambda ,\mu \}}{\max \{\lambda ,\mu \}+2}<2%
\text{;}  \label{e-gam-ineq3}
\end{equation}
\end{enumerate}
\end{lemma}

\begin{proof}
Note that $\frac{\partial \nu }{\partial \lambda }(\lambda ,\mu )=4\frac{%
(\mu -2)^{2}}{(4-\lambda \mu )^{2}}\geq 0$. Similarly, $\frac{\partial \nu }{%
\partial \mu }(\lambda ,\mu )=4\frac{(\lambda -2)^{2}}{(4-\lambda \mu )^{2}}%
\geq 0$. Suppose that $\lambda \leq \mu $. Note that $\lambda <2$ in this
case, because $\lambda \mu <4$. Thus, 
\begin{equation}
\lambda <\frac{4\lambda }{2+\lambda }=\nu (\lambda ,\lambda )\leq \nu
(\lambda ,\mu )\text{.}  \label{e-lam4}
\end{equation}%
Moreover, 
\begin{equation*}
\nu (\lambda ,\mu )\leq \nu (\mu ,\mu )=\frac{4\mu }{2+\mu }<2
\end{equation*}%
if $\lambda \leq \mu <2$. Now suppose that $\mu \leq \lambda $. In a similar
way as above one can prove that 
\begin{equation}
\mu <\frac{4\mu }{2+\mu }=\nu (\mu ,\mu )\leq \nu (\lambda ,\mu )\text{.}
\label{e-mi4}
\end{equation}%
Moreover, 
\begin{equation*}
\nu (\lambda ,\mu )\leq \nu (\lambda ,\lambda )=\frac{4\lambda }{2+\lambda }%
<2
\end{equation*}%
if $\mu \leq \lambda <2$. The considerations made above prove all of the
inequalities in (\ref{e-gam-ineq1})-(\ref{e-gam-ineq3}).
\end{proof}

\bigskip

Let $\lambda ,\mu >0$ be such that $\lambda \mu <4$. The theorem below shows
that composition of $\lambda $-RFNE and $\mu $-RFNE operators is $\nu
(\lambda ,\nu )$-RFNE. This result extends a well known result of Ogura and
Yamada \cite[Theorem 3(b)]{OY02} (cf. \cite[Theorem 2.2.37]{Ceg12}), where
the case $\lambda ,\mu \in (0,2)$ was considered. Moreover, we show that the
constant $\nu ^{\ast }:=\nu (\lambda ,\mu )$ is optimal. An equivalent
formulation of the first part of (ii) of the theorem below was presented in 
\cite[Prop. 2.5]{BDP22} in terms of conically averaged operators. However,
we give another proof of this part which we apply in Subsection \ref{ss-3.3}
in order to introduce extrapolations of compositions of RFNE operators and
extrapolations of compositions of relaxed cutters. As far as we know, the
second part of item (ii) in the theorem below as well as items (i) and
(iii)-(v) are new. For $x,y\in \mathcal{H}$ denote $a_{1}:=T(x)-x$, $%
a_{2}:=T(y)-y$, $b_{1}:=UT(x)-T(x)$ and $b_{2}:=UT(y)-T(y)$.

\begin{theorem}
\label{t-comp-RFNE}Let $\lambda ,\mu >0$, $T:\mathcal{H}\rightarrow \mathcal{%
H}$ be $\lambda $-RFNE, $U:\mathcal{H}\rightarrow \mathcal{H}$ be $\mu $%
-RFNE. Then

\begin{enumerate}
\item[$\mathrm{(i)}$] For arbitrary $x,y\in \mathcal{H}$ and for arbitrary $%
\nu \in 
%TCIMACRO{\U{211d} }%
%BeginExpansion
\mathbb{R}
%EndExpansion
\setminus \{0\}$ it holds that 
\begin{eqnarray}
&&\langle y-x,(UT(x)-x)-(UT(y)-y)\rangle -\frac{1}{\nu }\Vert
(UT(x)-x)-(UT(y)-y)\Vert ^{2}  \notag \\
&\geq &(\frac{1}{\lambda }-\frac{1}{\nu })\Vert a_{1}-a_{2}\Vert ^{2}+(\frac{%
1}{\mu }-\frac{1}{\nu })\Vert b_{1}-b_{2}\Vert ^{2}+(1-\frac{2}{\nu }%
)\langle a_{1}-a_{2},b_{1}-b_{2}\rangle \text{.}  \label{e-comp-RFNE1}
\end{eqnarray}

\item[$\mathrm{(ii)}$] If $\lambda \mu <4$ then $UT$ is $\nu ^{\ast }$-RFNE,
where $\nu ^{\ast }:=\nu (\lambda ,\mu )$ is given by $\mathrm{(\ref%
{e-gam-sol})}$. Consequently, $UT$ satisfies the demi closedness principle.
Moreover, if $\dim \mathcal{H}\geq 2$ then the constant $\nu ^{\ast }$ is
optimal, i.e. for arbitrary $\rho \in (0,\nu ^{\ast })$ there are a $\lambda 
$-RFNE operator $T$ and a $\mu $-RFNE operator $U$ such that $UT$ is not $%
\rho $-RFNE.

\item[$\mathrm{(iii)}$] If $\lambda \mu <\lambda +\mu $ and $\limfunc{Fix}%
T\cap \limfunc{Fix}U\neq \emptyset $ then $\limfunc{Fix}UT=\limfunc{Fix}%
T\cap \limfunc{Fix}U$.

\item[$\mathrm{(iv)}$] Suppose that $\lambda \mu >4$. If $\lambda \mu \leq
\lambda +\mu $ and $\dim \mathcal{H}\geq 2$ or if $\lambda \mu >\lambda +\mu 
$ then there are a $\lambda $-RFNE operator $T$ and a $\mu $-RFNE operator $%
U $ such that $UT$ is not RFNE.

\item[$\mathrm{(v)}$] If $\lambda \mu \geq \lambda +\mu $ then there are a $%
\lambda $-RFNE operator $T$ and a $\mu $-RFNE operator $U$ such that $%
\limfunc{Fix}UT\neq \limfunc{Fix}T\cap \limfunc{Fix}U\neq \emptyset $.
\end{enumerate}
\end{theorem}

\begin{proof}
(i) Let $\nu \in 
%TCIMACRO{\U{211d} }%
%BeginExpansion
\mathbb{R}
%EndExpansion
$ and $x,y\in \mathcal{H}$ be arbitrary. It follows from (\ref{e-RFNE}) that 
$\langle y-x,a_{1}-a_{2}\rangle \geq \frac{1}{\lambda }\Vert
a_{1}-a_{2}\Vert ^{2}$ and $\langle T(y)-T(x),b_{1}-b_{2}\rangle \geq \frac{1%
}{\mu }\Vert b_{1}-b_{2}\Vert ^{2}$. We have 
\begin{eqnarray}
&&\langle y-x,(UT(x)-x)-(UT(y)-y)\rangle -\frac{1}{\nu }\Vert
(UT(x)-x)-(UT(y)-y)\Vert ^{2}  \notag \\
&=&\langle y-x,(a_{1}+b_{1})-(a_{2}+b_{2})\rangle -\frac{1}{\nu }\Vert
(a_{1}+b_{1})-(a_{2}+b_{2})\Vert ^{2}  \label{e-2} \\
&=&\langle y-x,a_{1}-a_{2}\rangle +\langle
T(y)-T(x)+(a_{1}-a_{2}),b_{1}-b_{2}\rangle -\frac{1}{\nu }\Vert
(a_{1}+b_{1})-(a_{2}+b_{2})\Vert ^{2}  \label{e-3} \\
&\geq &\frac{1}{\lambda }\Vert a_{1}-a_{2}\Vert ^{2}+\frac{1}{\mu }\Vert
b_{1}-b_{2}\Vert ^{2}+\langle a_{1}-a_{2},b_{1}-b_{2}\rangle -\frac{1}{\nu }%
\Vert (a_{1}+b_{1})-(a_{2}+b_{2})\Vert ^{2}  \label{e-4} \\
&=&(\frac{1}{\lambda }-\frac{1}{\nu })\Vert a_{1}-a_{2}\Vert ^{2}+(\frac{1}{%
\mu }-\frac{1}{\nu })\Vert b_{1}-b_{2}\Vert ^{2}+(1-\frac{2}{\nu })\langle
a_{1}-a_{2},b_{1}-b_{2}\rangle \text{.}  \label{e-5}
\end{eqnarray}

(ii) Let $\lambda \mu <4$ and $\nu ^{\ast }=\nu (\lambda ,\mu )$ be defined
by (\ref{e-gam-sol}). By Lemma \ref{l-ni}(i), $\nu ^{\ast }\geq \max
\{\lambda ,\mu \}>0$, thus, $\frac{1}{\lambda }-\frac{1}{\nu ^{\ast }}\geq 0$
and $\frac{1}{\mu }-\frac{1}{\nu ^{\ast }}\geq 0$. Now Lemma \ref{l-ni} and
the properties of the inner product give 
\begin{eqnarray}
&&(\frac{1}{\lambda }-\frac{1}{\nu ^{\ast }})\Vert a_{1}-a_{2}\Vert ^{2}+(%
\frac{1}{\mu }-\frac{1}{\nu ^{\ast }})\Vert b_{1}-b_{2}\Vert ^{2}+(1-\frac{2%
}{\nu ^{\ast }})\langle a_{1}-a_{2},b_{1}-b_{2}\rangle  \notag \\
&=&\left\Vert \sqrt{\frac{1}{\lambda }-\frac{1}{\nu ^{\ast }}}%
(a_{1}-a_{2})\mp \sqrt{\frac{1}{\mu }-\frac{1}{\nu ^{\ast }}}%
(b_{1}-b_{2})\right\Vert ^{2}\geq 0\text{,}  \label{e-6}
\end{eqnarray}%
where the sign $\mp $ should be replaced by $-$ if $0<\nu ^{\ast }<2$ and by 
$+$ if $\nu ^{\ast }\geq 2$. By inequalities (\ref{e-2})-(\ref{e-6}), 
\begin{equation}
\langle y-x,(UT(x)-x)-(UT(y)-y)\rangle -\frac{1}{\nu ^{\ast }}\Vert
(UT(x)-x)-(UT(y)-y)\Vert ^{2}\geq 0\text{.}  \label{e-UT-RFNE}
\end{equation}%
By Proposition \ref{p-RFNE} this means that $UT$ is a $\nu ^{\ast }$-RFNE.
Consequently, $UT$ satisfies the demi closedness principle as a relaxation
of an NE operator.

Now suppose that $\dim \mathcal{H}\geq 2$. We prove that $\nu ^{\ast }$ is
optimal. Suppose for simplicity, that $\mathcal{H}=%
%TCIMACRO{\U{211d} }%
%BeginExpansion
\mathbb{R}
%EndExpansion
^{2}$. Denote $H:=\{x=(x_{1},x_{2})\in 
%TCIMACRO{\U{211d} }%
%BeginExpansion
\mathbb{R}
%EndExpansion
^{2}:x_{2}=0\}$ and $H_{k}:=\{x=(x_{1},x_{2})\in 
%TCIMACRO{\U{211d} }%
%BeginExpansion
\mathbb{R}
%EndExpansion
^{2}:x_{1}-kx_{2}=k\}$ and define $T:=(P_{H})_{\lambda }$ and $%
U_{k}:=(P_{H_{k}})_{\mu }$, $k\geq 0$. Clearly, $T$ is $\lambda $-RFNE, $%
U_{k}$ is $\mu $-RFNE and $z_{k}:=(k,0)$ is the only fixed point of $U_{k}T$%
. Let $\rho \in (0,$ $\nu ^{\ast })$. We prove that for some $k\geq 0$ the
operator $U_{k}T$ is not $\rho $-RFNE. For $x=(0,\xi )$ with $\xi \in 
%TCIMACRO{\U{211d} }%
%BeginExpansion
\mathbb{R}
%EndExpansion
$ we have $T(x)=(0,(1-\lambda )\xi )$, 
\begin{equation}
U_{k}T(x)-x=\left( \frac{\mu k[(1-\lambda )\xi +1]}{k^{2}+1},-\lambda \xi -%
\frac{\mu k^{2}[(1-\lambda )\xi +1]}{k^{2}+1}\right)   \label{e-UkT}
\end{equation}%
and 
\begin{equation}
\langle z_{k}-x,U_{k}T(x)-x\rangle =\frac{\mu k^{2}[(1-\lambda )\xi
+1](1+\xi )}{k^{2}+1}+\xi ^{2}\lambda \text{.}  \label{e-zk-x}
\end{equation}%
Define 
\begin{equation*}
f_{k}(\xi ,\rho ):=\rho \langle z_{k}-x,U_{k}Tx-x\rangle -\Vert
U_{k}T(x)-x\Vert ^{2}\text{.}
\end{equation*}%
In order to prove that for $\rho <\nu ^{\ast }$ and for some $k\geq 0$ the
operator $U_{k}T$ is not $\rho $-RFNE it is enough to show that $f(\xi ,\rho
):=\lim_{k\rightarrow \infty }f_{k}(\xi ,\rho )<0$ for some $\xi \in 
%TCIMACRO{\U{211d} }%
%BeginExpansion
\mathbb{R}
%EndExpansion
$ and for arbitrary $\rho \in (0,$ $\nu ^{\ast })$. Actually, it is enough
to show this for $\rho $ close to $\nu ^{\ast }$, because, if an operator $S$
is $\nu _{1}$-RFNE and $v_{2}>\nu _{1}$ then $S$ is $\nu _{2}$-RFNE. By (\ref%
{e-UkT})-(\ref{e-zk-x}), we have 
\begin{equation*}
f(\xi ,\rho )=(\lambda +\mu -\lambda \mu )[\rho -(\lambda +\mu -\lambda \mu
)]\xi ^{2}+[\rho (2-\lambda )-2(\lambda +\mu -\lambda \mu )]\mu \xi +\mu
\rho -\mu ^{2}\text{.}
\end{equation*}%
Note that $0<\lambda +\mu -\lambda \mu <\nu ^{\ast }$. If $\rho \in (\lambda
+\mu -\lambda \mu ,\nu ^{\ast })$ then the function $f(\cdot ,\rho )$
attains its minimum at 
\begin{equation}
\xi =\xi ^{\ast }(\rho )=-\frac{[\rho (2-\lambda )-2(\lambda +\mu -\lambda
\mu )]\mu }{2(\lambda +\mu -\lambda \mu )[\rho -(\lambda +\mu -\lambda \mu )]%
}  \label{e-ksi}
\end{equation}%
equal to 
\begin{equation}
h(\rho ):=f(\xi ^{\ast }(\rho ),\rho )=-\frac{\mu ^{2}}{4(\lambda +\mu
-\lambda \mu )}\frac{[\rho (2-\lambda )-2(\lambda +\mu -\lambda \mu )]^{2}}{%
\rho -(\lambda +\mu -\lambda \mu )}+\mu \rho -\mu ^{2}\text{.}
\label{e-h(ro)}
\end{equation}%
A direct calculus shows that $h(\nu ^{\ast })=0$ and that 
\begin{equation*}
\frac{\func{d}h}{\func{d}\rho }(\nu ^{\ast })=\frac{4-\lambda \mu }{\lambda
+\mu -\lambda \mu }>0\text{.}
\end{equation*}%
Thus, for $\rho <\nu ^{\ast }$ and sufficiently close to $\nu ^{\ast }$ we
have $f(\xi ^{\ast }(\rho ),\rho )<0$. This completes the proof of (ii).

(iii) Let $\lambda \mu <\lambda +\mu $ and $\limfunc{Fix}T\cap \limfunc{Fix}%
U\neq \emptyset $. If $\max \{\lambda ,\mu \}<2$ then $T$ and $U$ are SQNE
and it follows from \cite[Prop. 2.10]{BB96} that $\limfunc{Fix}UT=\limfunc{%
Fix}T\cap \limfunc{Fix}U$. Suppose that $\max \{\lambda ,\mu \}\geq 2$.
Lemma \ref{l-ni}(i) yields that $\nu ^{\ast }\geq 2$ if $\lambda \mu <4$.
Moreover $\nu ^{\ast }<0$ if $\lambda \mu >4$. The inclusion $\limfunc{Fix}%
T\cap \limfunc{Fix}U\subseteq \limfunc{Fix}UT$ is clear. We prove that 
\begin{equation}
\limfunc{Fix}UT\subseteq \limfunc{Fix}T\cap \limfunc{Fix}U\text{.}
\label{e-FixUT<}
\end{equation}%
The inclusion is obvious if $\limfunc{Fix}UT=\emptyset $. Suppose that $%
\limfunc{Fix}UT\neq \emptyset $ and that the opposite to (\ref{e-FixUT<})
holds true. Inequalities (\ref{e-2})-(\ref{e-6}) for $y\in \limfunc{Fix}UT$
give 
\begin{equation}
\langle y-x,UTx-x\rangle -\frac{1}{\nu ^{\ast }}\Vert UTx-x\Vert ^{2}\geq
\left\Vert \sqrt{\frac{1}{\lambda }-\frac{1}{\nu ^{\ast }}}(a_{1}-a_{2})+%
\sqrt{\frac{1}{\mu }-\frac{1}{\nu ^{\ast }}}(b_{1}-b_{2})\right\Vert
^{2}\geq 0  \label{e-UT-RFNE2}
\end{equation}%
(note that (\ref{e-6}) with the $+$ sign instead of $\mp $ is also true if $%
4<\lambda \mu <\lambda +\mu $, because $\nu ^{\ast }<0$ in this case). Let $%
x\in \limfunc{Fix}UT$ be such that $x\notin \limfunc{Fix}T\cap \limfunc{Fix}U
$. Then it follows from (\ref{e-UT-RFNE2}) that 
\begin{equation}
\sqrt{\frac{1}{\lambda }-\frac{1}{\nu ^{\ast }}}(a_{1}-a_{2})+\sqrt{\frac{1}{%
\mu }-\frac{1}{\nu ^{\ast }}}(b_{1}-b_{2})=0
\end{equation}%
and 
\begin{equation*}
(a_{1}-a_{2})+(b_{1}-b_{2})=a_{1}+b_{1}=UTx-x=0\text{.}
\end{equation*}%
This leads to $\sqrt{\frac{1}{\lambda }-\frac{1}{\nu ^{\ast }}}=\sqrt{\frac{1%
}{\mu }-\frac{1}{\nu ^{\ast }}}$, consequently, $\lambda =\mu \geq 2$. This
stands in contradiction to the assumption $\lambda \mu <\lambda +\mu $,
which proves (\ref{e-FixUT<}).

(iv) We consider 2 cases:

(a) $4<\lambda \mu \leq \lambda +\mu $ and $\dim \mathcal{H}\geq 2$. For
simplicity we suppose that $\mathcal{H}=%
%TCIMACRO{\U{211d} }%
%BeginExpansion
\mathbb{R}
%EndExpansion
^{2}$. For the operators $T$ and $U_{k}$ we use the same notation as in
(ii). Let $x=(0,\xi )$ with $\xi \in 
%TCIMACRO{\U{211d} }%
%BeginExpansion
\mathbb{R}
%EndExpansion
$. Then $U_{k}T(x)$ is given by (\ref{e-UkT}). Suppose that for all $k\geq 0$
the operators $U_{k}T$ are RFNE, consequently, they are relaxed cutters.
Then $\langle z_{k}-x,U_{k}T(x)-x\rangle \geq \frac{1}{\nu _{k}}\Vert
U_{k}T(x)-x\Vert ^{2}\geq 0$ for some $\nu _{k}>0$, $k\geq 0$, and 
\begin{eqnarray*}
\alpha &:&=\lim_{k}\langle z_{k}-x,U_{k}Tx-x\rangle =\lim_{k}\frac{\mu
k^{2}[(1-\lambda )\xi +1](1+\xi )}{k^{2}+1}+\xi ^{2}\lambda \\
&=&\mu \lbrack (1-\lambda )\xi +1](1+\xi )+\xi ^{2}\lambda =(\lambda +\mu
-\lambda \mu )\xi ^{2}+(2-\lambda )\mu \xi +\mu \geq 0
\end{eqnarray*}%
for all $\xi \in 
%TCIMACRO{\U{211d} }%
%BeginExpansion
\mathbb{R}
%EndExpansion
$. If $\lambda \mu =\lambda +\mu $ then we obtain $\alpha <0$ for some $\xi $%
, a contradiction. If $\lambda \mu <\lambda +\mu $ then 
\begin{equation*}
\Delta :=(2-\lambda )^{2}\mu ^{2}-4\mu (\lambda +\mu -\lambda \mu )=\lambda
\mu (\lambda \mu -4)>0\text{.}
\end{equation*}%
Thus, for $\xi =\frac{(\lambda -2)\mu }{2(\lambda +\mu -\lambda \mu )}$ we
obtain $\alpha <0$, a contradiction. Thus, there is $k\geq 1$ for which $%
\langle z_{k}-x,U_{k}Tx-x\rangle <0$. This means that $U_{k}T$ is not a
relaxed cutter, thus $U_{k}T$ is not RFNE.

(b) $\lambda \mu >\lambda +\mu $. Then $\lambda >1$. Let $H\subseteq 
\mathcal{H}$ be a hyperplane. Denote $P:=P_{H}$ and define $T:=P_{\lambda }$
and $U:=P_{\mu }$. Clearly, $T$ is $\lambda $-RFNE, $U$ is $\mu $-RFNE and $%
\limfunc{Fix}UT=\limfunc{Fix}T=\limfunc{Fix}U=H$. Let $x\in \mathcal{H}$ and 
$z\in H$ be arbitrary. We have $y:=P_{\lambda }(x)=x+\lambda (P(x)-x)$. By
the properties of the metric projection, $P(y)=P(x)$. Thus, 
\begin{eqnarray*}
UT(x)-x &=&P_{\mu }P_{\lambda }(x)-x=y+\mu (P(y)-y)-x \\
&=&x+\lambda (P(x)-x)+\mu (1-\lambda )(P(x)-x)-x \\
&=&(\lambda +\mu -\lambda \mu )(P(x)-x)
\end{eqnarray*}%
and 
\begin{eqnarray*}
\langle z-x,UT(x)-x\rangle &=&(\lambda +\mu -\lambda \mu )\langle
z-x,P(x)-x\rangle \\
&=&(\lambda +\mu -\lambda \mu )\Vert P(x)-x\Vert ^{2}<0
\end{eqnarray*}%
if $x\notin H$, which means that $UT$ is not a relaxed cutter, thus $UT$ is
not RFNE.

(v) Let $\lambda \mu \geq \lambda +\mu $. Then $\lambda >1$. Let $H\subseteq 
\mathcal{H}$ be a hyperplane. Denote $P:=P_{H}$, $\sigma :=\frac{\lambda }{%
\mu (\lambda -1)}$ and define $T:=P_{\lambda }$ and $U:=(P_{\sigma })_{\mu
}=P_{\sigma \mu }$. Clearly, $T$ is $\lambda $-RFNE, $U$ is $\sigma \mu $%
-RFNE and $\limfunc{Fix}T=\limfunc{Fix}U=H$. Similarly as in the proof of
case (b) of (iv), for any $x\in \mathcal{H}$ we have%
\begin{eqnarray*}
UT(x)-x &=&P_{\sigma \mu }P_{\lambda }(x)-x=(\lambda +\sigma \mu -\lambda
\sigma \mu )(P(x)-x) \\
&=&(\lambda +\frac{\lambda }{\lambda -1}-\frac{\lambda ^{2}}{\lambda -1}%
)(P(x)-x)=0
\end{eqnarray*}%
which means that $\limfunc{Fix}UT=\mathcal{H}\neq H=\limfunc{Fix}T\cap 
\limfunc{Fix}U$.
\end{proof}

\bigskip

By Proposition \ref{c-SPC-RFNE}, there is equivalence between $\alpha $-SPC
and $\lambda $-RFNE operators, where $\alpha \in (-\infty ,1)$ and $\lambda
=2/(1-\alpha )\in (0,+\infty )$. Thus, Theorem \ref{t-comp-RFNE} can be
presented in terms of SPC instead of RFNE operators. Note that for $\alpha
,\beta \in (-\infty ,1)$, the inequality $\alpha +\beta <\alpha \beta $
implies that $\alpha +\beta <0$, consequently at least one of $\alpha ,\beta 
$ is negative and $\frac{\alpha \beta }{\alpha +\beta }<1$ (see Appendix).

\begin{corollary}
\label{c-comp-SPC}Let $\alpha ,\beta \in (-\infty ,1)$, $T:\mathcal{H}%
\rightarrow \mathcal{H}$ be an $\alpha $-SPC, $U:\mathcal{H}\rightarrow 
\mathcal{H}$ be a $\beta $-SPC.

\begin{enumerate}
\item[$\mathrm{(i)}$] If $\alpha +\beta <\alpha \beta $ then $UT$ is a $%
\gamma ^{\ast }$-SPC, where 
\begin{equation}
\gamma ^{\ast }=\gamma (\alpha ,\beta )=\frac{\alpha \beta }{\alpha +\beta }%
\text{.}  \label{e-gamma*}
\end{equation}

Consequently, $UT$ satisfies the demi closedness principle. If, additionally,

\begin{enumerate}
\item[$\mathrm{(a)}$] $\alpha ,\beta <0$ then $\gamma ^{\ast }<0$;

\item[$\mathrm{(b)}$] $\alpha \beta <0$ then $\gamma ^{\ast }\in (0,1)$.
\end{enumerate}

Moreover, if $\dim \mathcal{H}\geq 2$ then the constant $\gamma ^{\ast }$ is
optimal, i.e. for arbitrary $\rho \in (-\infty ,\gamma ^{\ast })$ there are
an $\alpha $-SPC\quad $T$ and a $\beta $-SPC\quad $U$ such that $UT$ is not
a $\rho $-SPC.

\item[$\mathrm{(ii)}$] If $\alpha +\beta <0$ and $\limfunc{Fix}T\cap 
\limfunc{Fix}U\neq \emptyset $ then $\limfunc{Fix}UT=\limfunc{Fix}T\cap 
\limfunc{Fix}U$.

\item[$\mathrm{(iii)}$] Suppose that $\alpha +\beta >\alpha \beta $. If $%
\alpha +\beta \leq 0$ and $\dim \mathcal{H}\geq 2$ or if $\alpha +\beta >0$
then there are an $\alpha $-SPC\quad $T$ and a $\beta $-SPC \quad $U$ such
that $UT$ is not an SPC.

\item[$\mathrm{(iv)}$] If $\alpha +\beta \geq 0$ then there are an $\alpha $%
-SPC\quad $T$ and a $\beta $-SPC\quad $U$ such that $\limfunc{Fix}UT\neq 
\limfunc{Fix}T\cap \limfunc{Fix}U\neq \emptyset $.
\end{enumerate}
\end{corollary}

Let $\lambda ,\mu >0$ be such that $\lambda \mu <4$, $T:\mathcal{H}%
\rightarrow \mathcal{H}$ be $\lambda $-RFNE, $U:\mathcal{H}\rightarrow 
\mathcal{H}$ be $\mu $-RFNE (equivalently, let $T:\mathcal{H}\rightarrow 
\mathcal{H}$ be an $\alpha $-SPC, $U:\mathcal{H}\rightarrow \mathcal{H}$ be
a $\beta $-SPC, where $\alpha ,\beta \in (-\infty ,1)$ are such that $\alpha
+\beta <\alpha \beta $) with $\limfunc{Fix}(UT)\neq \emptyset $. Theorem \ref%
{t-comp-RFNE} (equivalently, Corollary \ref{c-comp-SPC}) together with
Proposition \ref{p-KM} and the fact that $(UT)_{1/\nu ^{\ast }}$ is FNE
yield the weak convergence of sequences $x^{k+1}=(UT)_{\lambda _{k}/\nu
^{\ast }}(x^{k})$ to some $x^{\ast }\in \limfunc{Fix}(UT)$, where $\lambda
_{k}\in \lbrack \varepsilon -2-\varepsilon ]$ for some small $\varepsilon >0$%
. In Section \ref{s-4} we show that an enlarged range of $\lambda _{k}$
guarantees the weak convergence (see Theorem \ref{t-iter-RFNE} and
Corollaries \ref{c-iter-SPCa} and \ref{c-iter-SPC}).

\bigskip

If we set $y\in \limfunc{Fix}T$ in Definition \ref{d-NE}(e) then we receive
the definition of an $\alpha $-demicontraction (equivalently a $\lambda $%
-relaxed cutter, see Corollary \ref{c-dc-c}) Thus, Theorem \ref{t-comp-RFNE}
yields a part of the following result which extends a well known result of
Yamada and Ogura \cite[Proposition 1(d)]{YO04} (cf. \cite[Theorem 2.1.46]%
{Ceg12}), where only the case $\lambda ,\mu \in (0,2)$ was considered. The
optimality of the constant $\nu ^{\ast }$ which was proved in Theorem \ref%
{t-comp-RFNE} also applies for composition of relaxed cutters. As far as we
know, the results presented in the theorem below are new. Denote $a:=T(x)-x$
and $b:=UT(x)-T(x)$.

\begin{theorem}
\label{t-comp-dc}Let $\lambda ,\mu >0$, $T:\mathcal{H}\rightarrow \mathcal{H}
$ be a $\lambda $-relaxed cutter, $U:\mathcal{H}\rightarrow \mathcal{H}$ be
a $\mu $-relaxed cutter and let $\limfunc{Fix}T\cap \limfunc{Fix}U\neq
\emptyset $. Then

\begin{enumerate}
\item[$\mathrm{(i)}$] For arbitrary $x\in \mathcal{H}$ and $z\in \limfunc{Fix%
}T\cap \limfunc{Fix}U$ and for arbitrary $\nu \in 
%TCIMACRO{\U{211d} }%
%BeginExpansion
\mathbb{R}
%EndExpansion
$ it holds that 
\begin{equation}
\langle z-x,UT(x)-x\rangle -\frac{1}{\nu }\Vert UT(x)-x\Vert ^{2}\geq (\frac{%
1}{\lambda }-\frac{1}{\nu })\Vert a\Vert ^{2}+(\frac{1}{\mu }-\frac{1}{\nu }%
)\Vert b\Vert ^{2}+(1-\frac{2}{\nu })\langle a,b\rangle \text{.}
\label{e-comp-dc}
\end{equation}

\item[$\mathrm{(ii)}$] If $\lambda \mu <4$ then $UT$ is a $\nu ^{\ast }$%
-relaxed cutter, where $\nu ^{\ast }=\nu (\lambda ,\mu )$ is given by $%
\mathrm{(\ref{e-gam-sol})}$. Moreover, if $\dim \mathcal{H}\geq 2$ then the
constant $\nu ^{\ast }$ is optimal, i.e. for arbitrary $\rho \in (0,\nu
^{\ast })$ there are a $\lambda $-relaxed cutter $T$ and a $\mu $-relaxed
cutter $U$ such that $UT$ is not a $\rho $-relaxed cutter.

\item[$\mathrm{(iii)}$] If $\lambda \mu <\lambda +\mu $ then $\limfunc{Fix}%
UT=\limfunc{Fix}T\cap \limfunc{Fix}U$.

\item[$\mathrm{(iv)}$] Suppose that $\lambda \mu >4$. If $\lambda \mu \leq
\lambda +\mu $ and $\dim \mathcal{H}\geq 2$ or if $\lambda \mu >\lambda +\mu 
$ then there are a $\lambda $-relaxed cutter $T$ and a $\mu $-relaxed cutter 
$U$ such that $UT$ is not a relaxed cutter.

\item[$\mathrm{(v)}$] If $\lambda \mu \geq \lambda +\mu $ then there are a $%
\lambda $-relaxed cutter $T$ and a $\mu $-relaxed cutter $U$ such that $%
\limfunc{Fix}UT\neq \limfunc{Fix}T\cap \limfunc{Fix}U$.

\item[$\mathrm{(vi)}$] Suppose that $T$ and $U$ satisfy the demi closedness
principle. If $4\neq \lambda \mu <\lambda +\mu $ then the operators $UT$ and 
$(UT)_{1/\nu ^{\ast }}$ also satisfy the demi closedness principle.
\end{enumerate}
\end{theorem}

\begin{proof}
Let $x\in \mathcal{H}$ and $z\in \limfunc{Fix}T\cap \limfunc{Fix}U$ be
arbitrary. Denote $a:=T(x)-x$ and $b:=UT(x)-T(x)$. It follows from (\ref%
{e-RC}) that $\langle z-x,a\rangle \geq \frac{1}{\lambda }\Vert a\Vert ^{2}$
and $\langle z-T(x),b\rangle \geq \frac{1}{\mu }\Vert b\Vert ^{2}$. Thus,
for relaxed cutters $T$ and $U$ inequalities (\ref{e-2})-(\ref{e-5}) hold
for $y=z$ which proves (i). This and inequality (\ref{e-6}) gives 
\begin{equation}
\langle z-x,UT(x)-x\rangle -\frac{1}{\nu ^{\ast }}\Vert UT(x)-x\Vert
^{2}\geq \left\Vert \sqrt{\frac{1}{\lambda }-\frac{1}{\nu ^{\ast }}}a\mp 
\sqrt{\frac{1}{\mu }-\frac{1}{\nu ^{\ast }}}b\right\Vert ^{2}\geq 0\text{,}
\label{e-UT-DC}
\end{equation}%
where the sign $\mp $ should be replaced by $-$ if $0<\nu ^{\ast }<2$ and by 
$+$ if $\nu ^{\ast }\geq 2$ or $\nu ^{\ast }<0$. The proof of (ii)-(v) is
similar to the proof of Theorem \ref{t-comp-RFNE}.

(vi) Let $4\neq \lambda \mu <\lambda +\mu $. If $\max \{\lambda ,\mu \}<2$
then $T,U$ are SQNE and $UT$ satisfies the demi closedness principle \cite[%
Th. 4.2]{Ceg15}. Suppose that $\max \{\lambda ,\mu \}\geq 2$. Then $\nu
(\lambda ,\mu )\geq 2$ or $\nu (\lambda ,\mu )<0$, as it was observed
before. In this case, the sign $\mp $ in (\ref{e-UT-DC}) should be replaced
by $+$. Let $z\in \limfunc{Fix}T$ be fixed, $\{x^{k}\}_{k=0}^{\infty }$ be a
bounded sequence with $\Vert UT(x^{k})-x^{k}\Vert \rightarrow 0$, $y\in 
\mathcal{H}$ be its weak cluster point and let $\{x^{n_{k}}\}_{k=0}^{\infty }
$ be its subsequence such that $x^{n_{k}}\rightharpoonup y$. We prove that $%
y\in \limfunc{Fix}UT$. Note that $T_{\lambda ^{-1}}$ is a cutter and $%
\limfunc{Fix}T_{\lambda ^{-1}}=\limfunc{Fix}T$. By \cite[Cor. 2.1.37]{Ceg12}%
, 
\begin{equation*}
\Vert T(x^{k})-x^{k}\Vert =\lambda \Vert T_{\lambda ^{-1}}(x^{k})-x^{k}\Vert
\leq \lambda \Vert P_{\limfunc{Fix}T}(x^{k})-x^{k}\Vert \leq \lambda \Vert
P_{\limfunc{Fix}T}(x^{k})-z\Vert \leq \lambda \Vert x^{k}-z\Vert \text{,}
\end{equation*}%
i.e. $\Vert T(x^{k})-x^{k}\Vert $ is bounded. Similarly as before, denote $%
a^{k}:=T(x^{k})-x^{k}$ and $b^{k}:=UT(x^{k})-T(x^{k})$. Let $%
\{x^{m_{k}}\}_{k=0}^{\infty }\subseteq \{x^{n_{k}}\}_{k=0}^{\infty }$ be
such that $\Vert a^{m_{k}}\Vert \rightarrow \alpha $ for some $\alpha \geq 0$%
. We prove that $\alpha =0$. Suppose that the opposite holds, i.e. $\alpha >0
$. By setting $x=x^{m_{k}}$ in (\ref{e-UT-DC}), we obtain%
\begin{equation*}
0\leq \left\Vert \sqrt{\frac{1}{\lambda }-\frac{1}{\nu ^{\ast }}}a^{m_{k}}+%
\sqrt{\frac{1}{\mu }-\frac{1}{\nu ^{\ast }}}b^{m_{k}}\right\Vert ^{2}\leq
\langle z-x^{m_{k}},UT(x^{m_{k}})-x^{m_{k}}\rangle -\frac{1}{\nu }\Vert
UT(x^{m_{k}})-x^{m_{k}}\Vert ^{2}\rightarrow 0\text{,}
\end{equation*}%
consequently, 
\begin{equation*}
\left\Vert \sqrt{\frac{1}{\lambda }-\frac{1}{\nu ^{\ast }}}a^{m_{k}}+\sqrt{%
\frac{1}{\mu }-\frac{1}{\nu ^{\ast }}}b^{m_{k}}\right\Vert \rightarrow 0
\end{equation*}%
Because $a^{k}+b^{k}=UT(x^{k})-x^{k}\rightarrow 0$, it holds $%
b^{k}=-a^{k}+d^{k}$ with $d^{k}\rightarrow 0$, thus $\Vert b^{m_{k}}\Vert
\rightarrow \alpha $. By the triangle inequality, 
\begin{equation*}
\left\Vert \sqrt{\frac{1}{\lambda }-\frac{1}{\nu ^{\ast }}}a^{m_{k}}+\sqrt{%
\frac{1}{\mu }-\frac{1}{\nu ^{\ast }}}b^{m_{k}}\right\Vert \geq \left\vert 
\sqrt{\frac{1}{\lambda }-\frac{1}{\nu ^{\ast }}}-\sqrt{\frac{1}{\mu }-\frac{1%
}{\nu ^{\ast }}}\right\vert \Vert a^{m_{k}}\Vert -\sqrt{\frac{1}{\mu }-\frac{%
1}{\nu ^{\ast }}}\Vert d^{m_{k}}\Vert \text{.}
\end{equation*}%
This yields 
\begin{equation*}
\left\vert \sqrt{\frac{1}{\lambda }-\frac{1}{\nu ^{\ast }}}-\sqrt{\frac{1}{%
\mu }-\frac{1}{\nu ^{\ast }}}\right\vert \Vert a^{m_{k}}\Vert \rightarrow 0%
\text{,}
\end{equation*}%
consequently, $\sqrt{\frac{1}{\lambda }-\frac{1}{\nu ^{\ast }}}=\sqrt{\frac{1%
}{\mu }-\frac{1}{\nu ^{\ast }}}$ which yields $\lambda =\mu \geq 2$, a
contradiction to the assumption $\lambda \mu <\lambda +\mu $. Thus, $\alpha
=0$, i.e. $\Vert T(x^{m_{k}})-x^{m_{k}}\Vert =\Vert a^{m_{k}}\Vert
\rightarrow 0$ and $\Vert UT(x^{m_{k}})-T(x^{m_{k}})\Vert =\Vert
b^{m_{k}}\Vert \rightarrow 0$. Because $T$ satisfies the demi closedness
principle, $y\in \limfunc{Fix}T$. Moreover $%
y^{m_{k}}:=T(x^{m_{k}})=a^{m_{k}}+x^{m_{k}}\rightharpoonup y$. Because $U$
satisfies the demi closedness principle, $y\in \limfunc{Fix}U$. Thus, $y\in 
\limfunc{Fix}UT$, i.e. $UT$ satisfies the demi closedness principle.
Consequently, $(UT)_{1/\nu ^{\ast }}$ also satisfies the demi closedness
principle as a relaxation of an operator satisfying the demi closedness
principle.
\end{proof}

\bigskip

Theorem \ref{t-comp-dc} can be presented in terms of demicontractions
instead of relaxed cutters.

\begin{corollary}
\label{c-comp-dc}Let $\alpha ,\beta \in (-\infty ,1)$, $T:\mathcal{H}%
\rightarrow \mathcal{H}$ be an $\alpha $-demicontraction, $U:\mathcal{H}%
\rightarrow \mathcal{H}$ be a $\beta $-demicontraction, and let $\limfunc{Fix%
}T\cap \limfunc{Fix}U\neq \emptyset $.

\begin{enumerate}
\item[$\mathrm{(i)}$] If $\alpha +\beta <\alpha \beta $ then $UT$ is a $%
\gamma ^{\ast }$-demicontraction, where $\gamma ^{\ast }$ is given by $%
\mathrm{(\ref{e-gamma*})}$.

If, additionally,

\begin{enumerate}
\item[$\mathrm{(a)}$] $\alpha ,\beta <0$ then $\gamma ^{\ast }<0$ and $UT$
is $-\gamma ^{\ast }$-SQNE;

\item[$\mathrm{(b)}$] $\alpha \beta <0$ then $\gamma ^{\ast }\in (0,1)$.
\end{enumerate}

Moreover, if $\dim \mathcal{H}\geq 2$ then the constant $\gamma ^{\ast }$ is
optimal, i.e. for arbitrary $\rho \in (-\infty ,\gamma ^{\ast })$ there are
an $\alpha $-demicontraction $T$ and a $\beta $-demicontraction $U$ such
that $UT$ is not a $\rho $-demicontraction.

\item[$\mathrm{(ii)}$] If $\alpha +\beta <0$ then $\limfunc{Fix}UT=\limfunc{%
Fix}T\cap \limfunc{Fix}U$.

\item[$\mathrm{(iii)}$] Suppose that $\alpha +\beta >\alpha \beta $. If $%
\alpha +\beta \leq 0$ and $\dim \mathcal{H}\geq 2$ or if $\alpha +\beta >0$
then there are an $\alpha $-demicontraction $T$ and a $\beta $%
-demicontraction $U$ such that $UT$ is not a demicontraction.

\item[$\mathrm{(iv)}$] If $\alpha +\beta \geq 0$ then there are an $\alpha $%
-demicontraction $T$ and a $\beta $-demicontraction $U$ such that $\limfunc{%
Fix}UT\neq \limfunc{Fix}T\cap \limfunc{Fix}U$.

\item[$\mathrm{(v)}$] Suppose that $T$ and $U$ satisfy the demi closedness
principle. If $\alpha +\beta <0$ then the operators $UT$ and $(UT)_{1/\nu
^{\ast }}$ also satisfy the demi closedness principle.
\end{enumerate}
\end{corollary}

Let $\lambda ,\mu >0$ be such that $\lambda \mu <4$, $T:\mathcal{H}%
\rightarrow \mathcal{H}$ be a $\lambda $-relaxed cutter, $U:\mathcal{H}%
\rightarrow \mathcal{H}$ be a $\mu $-relaxed cutter (equivalently, let $T:%
\mathcal{H}\rightarrow \mathcal{H}$ be an $\alpha $-demicontraction, $U:%
\mathcal{H}\rightarrow \mathcal{H}$ be a $\beta $-demicontraction, where $%
\alpha ,\beta \in (-\infty ,1)$ are such that $\alpha +\beta <\alpha \beta $%
) with $\limfunc{Fix}T\cap \limfunc{Fix}U\neq \emptyset $, both satisfying
the demi closedness principle. Theorem \ref{t-comp-dc} (equivalently,
Corollary \ref{c-comp-dc}) together with Proposition \ref{p-WC} and the fact
that $(UT)_{1/\nu ^{\ast }}$ is a cutter yield the weak convergence of
sequences $x^{k+1}=(UT)_{\lambda _{k}/\nu ^{\ast }}(x^{k})$ to some $x^{\ast
}\in \limfunc{Fix}T\cap \limfunc{Fix}U$, where $\lambda _{k}\in \lbrack
\varepsilon -2-\varepsilon ]$ for some small $\varepsilon >0$. In Section %
\ref{s-4} we show that an enlarged range of $\lambda _{k}$ guarantees the
weak convergence (see Theorem \ref{t-iter-UV} and Corollary \ref{c-iter-UV}).

\bigskip

Let $T_{i}:\mathcal{H}\rightarrow \mathcal{H}$ be an $\alpha _{i}$%
-demicontraction where $\alpha _{i}\in (-\infty ,1)\smallsetminus \{0\}$, $%
i\in I:=\{1,2,...,m\}$. Define $T:=T_{m}T_{m-1}...T_{1}$ and denote 
\begin{equation}
\gamma _{k}:=(\sum_{i=1}^{k}\alpha _{i}^{-1})^{-1}\text{ and }\beta
_{k}:=(\sum_{i=k}^{m}\alpha _{i}^{-1})^{-1}\text{,}  \label{e-gam_k}
\end{equation}%
$k=1,2,...,m$. The following Theorem extends \cite[Th. 2.1.48]{Ceg12}, where
the case $\alpha _{i}<0$, $i\in I$, is considered.

\begin{theorem}
Let $\alpha _{i}\in (-\infty ,1)\smallsetminus \{0\}$, $i\in I$, where $%
\alpha _{i}>0$ for at most one $i\in I$. Suppose that $T_{i}:\mathcal{H}%
\rightarrow \mathcal{H}$, $i\in I$, are $\alpha _{i}$-demicontractions that
share a common fixed point. If $\gamma _{m}<1$ then the operator $T$ is a $%
\gamma _{m}$-demicontraction and $\limfunc{Fix}T=\bigcap_{i=1}^{m}\limfunc{%
Fix}T_{i}$. If, moreover, $T_{i}$, $i\in I$, satisfy the demi closedness
principle, then the operator $T$ also satisfies the demi closedness
principle.
\end{theorem}

\begin{proof}
Consider two cases:

(a) $\alpha _{i}<0$ for all $i\in I$. Then $T_{i}$ is $-\alpha _{i}$-SQNE, $%
i\in I$, and it follows from \cite[Th. 2.1.48]{Ceg12} that $T$ is $-\gamma
_{m}$-SQNE which means that $T$ is a $\gamma _{m}$-demicontraction.
Moreover, \cite[Prop. 2.10(i)]{BB96} yields $\limfunc{Fix}T=\bigcap_{i=1}^{m}%
\limfunc{Fix}T_{i}$. If all $T_{i}$, $i\in I$, satisfy the demi closedness
principle, then \cite[Th. 4.2]{Ceg15} yields that $T$ also satisfies the
demi closedness principle.

(b) $\alpha _{j}>0$ for some $j\in I$. Then $\alpha _{i}<0$ for all $i\neq j$%
. Suppose that $\gamma _{m}<1$. Denote 
\begin{equation*}
U_{i}:=T_{m}...T_{i}\text{ and }V_{i}:=T_{i}...T_{1}\text{.}
\end{equation*}%
We have 
\begin{equation*}
U_{j}=U_{j+1}T_{j},V_{j}=T_{j}V_{j-1}\text{ and }T=U_{j+1}V_{j}\text{.}
\end{equation*}%
Note that $\gamma _{j-1}<0$, $\beta _{j+1}<0$ and 
\begin{equation*}
\frac{1}{\gamma _{j}}+\frac{1}{\beta _{j+1}}=\frac{1}{\gamma _{j-1}}+\frac{1%
}{\alpha _{j}}+\frac{1}{\beta _{j+1}}=\frac{1}{\gamma _{m}}>1\text{.}
\end{equation*}%
Thus, 
\begin{equation}
\frac{1}{\gamma _{j-1}}+\frac{1}{\alpha _{j}}=\frac{1}{\gamma _{j}}>1\text{.}
\label{e-1/gamma}
\end{equation}%
By (a), $U_{j+1}$ is a $\beta _{j+1}$-demicontraction with $\beta _{j+1}<0$, 
$\limfunc{Fix}U_{j+1}=\bigcap_{i=j+1}^{m}\limfunc{Fix}T_{i}$, $V_{j-1}$ is a 
$\gamma _{j-1}$-demicontraction with $\gamma _{j-1}<0$ and $\limfunc{Fix}%
V_{j-1}=\bigcap_{i=1}^{j-1}\limfunc{Fix}T_{i}$. If $T_{i},$ $i=1,2,...,j-1$,
satisfy the demi closedness principle, then \cite[Th. 4.2]{Ceg15} yields
that $V_{j-1}$ also satisfies the demi closedness principle. Because $\gamma
_{j-1}\alpha _{j}<0$, (\ref{e-1/gamma}) yields 
\begin{equation}
\gamma _{j-1}+\alpha _{j}<\gamma _{j-1}\alpha _{j}<0.  \label{e-gam-j-1}
\end{equation}%
Now, Corollary \ref{c-comp-dc}(i)-(ii) yields that $V_{j}$ is a $\gamma _{j}$%
-demicontraction with $\gamma _{j}\in (0,1)$ and 
\begin{equation*}
\limfunc{Fix}V_{j}=\limfunc{Fix}T_{j}\cap \limfunc{Fix}V_{j-1}=%
\bigcap_{i=1}^{j}\limfunc{Fix}T_{i}\text{.}
\end{equation*}%
If $T_{i},$ $i=1,2,...,j$, satisfy the demi closedness principle, then
Corollary \ref{c-comp-dc}(v) yields that $V_{j}$ also satisfies the demi
closedness principle as composition of $T_{j}$ and $V_{j-1}$ satisfying (\ref%
{e-gam-j-1}). Because $\beta _{j+1}\gamma _{j}<0$, it holds 
\begin{equation}
\beta _{j+1}+\gamma _{j}<\beta _{j+1}\gamma _{j}<0.  \label{e-beta-j+1}
\end{equation}%
Applying Corollary \ref{c-comp-dc}(i)-(ii) again, we obtain that $%
T=U_{j+1}V_{j}$ is a $\gamma _{m}$-demicontraction and 
\begin{equation*}
\limfunc{Fix}T=\limfunc{Fix}U_{j+1}\cap \limfunc{Fix}V_{j}=%
\bigcap_{i=j+1}^{m}\limfunc{Fix}T_{i}\cap \bigcap_{i=1}^{j}\limfunc{Fix}%
T_{i}=\bigcap_{i=1}^{m}\limfunc{Fix}T_{i}\text{.}
\end{equation*}%
Finally, if $T_{i},$ $i\in I$, satisfy the demi closedness principle, then 
\cite[Th. 4.2]{Ceg15} yields that $U_{j+1}$ also satisfies the demi
closedness principle and Corollary \ref{c-comp-dc}(v) yields that $T$ also
satisfies the demi closedness principle as composition of $U_{j+1}$ and $%
V_{j}$ satisfying (\ref{e-beta-j+1}).
\end{proof}

\subsection{Extrapolation of composition of demicontractions\label{ss-3.3}}

Let $\lambda ,\mu >0$ be such that $\lambda \mu <4$. In Theorem \ref%
{t-comp-dc} we proved that for composition of $\lambda $- and $\mu $-relaxed
cutters $T$ and $U$, the operator $UT$ is a $\nu ^{\ast }$-relaxed cutter,
where the constant $\nu ^{\ast }$ is given by (\ref{e-gam-sol}). Moreover,
we proved that for $\rho \in (0,\nu ^{\ast })$ the operator $UT$ need not to
be a $\rho $-relaxed cutter. It turns out that if we allow $\rho $ to depend
on $T,U$ and $x$ then we can decrease $\rho $ for which $UT$ is a $\rho $%
-relaxed cutter. Applying this property in corresponding algorithms we are
able to enlarge the step size $\Vert x^{k+1}-x^{k}\Vert $ without loss of
the Fej\'{e}r monotonicity of $\{x^{k}\}_{k=0}^{\infty }$. This can lead to
a faster convergence of $x^{k}$ to a solution. We start with the definition
of a generalized relaxation of an operator.

\begin{definition}
%TCIMACRO{\TeXButton{rm}{\rm\ }}%
%BeginExpansion
\rm\ %
%EndExpansion
Let $\sigma :\mathcal{H}\rightarrow (0,+\infty )$ and let $S:\mathcal{H}%
\rightarrow \mathcal{H}$ be an operator. We call the operator $S_{\sigma }:%
\mathcal{H}\rightarrow \mathcal{H}$ defined by 
\begin{equation*}
S_{\sigma }(x):=x+\sigma (x)(S(x)-x)
\end{equation*}%
a \textit{generalized relaxation} of $S$ and $\sigma $ is called a \textit{%
relaxation function}. If $\sigma (x)\geq 1$ for all $x\in \mathcal{H}$ and $%
\sigma (x)>1$ for at least one $x\notin \limfunc{Fix}S$,  then $S_{\sigma }$
is called an \textit{extrapolation} of $S$ and $\sigma $ is called an 
\textit{extrapolation function}. If $S$ is a cutter then $S_{\sigma }$ is
called a \textit{generalized }$\sigma $\textit{-relaxed cutter}.
\end{definition}

If the function $\sigma $ is constant then the above definition coincides
with the classical definition of a relaxation. Clearly, in the case of
generalized relaxation, the constant $\lambda $ in inequality (\ref{e-RC})
should be replaced by $\sigma (x)$.

Extrapolations of FNE operators as well as of cutters have been successfully
applied in many papers, e.g. in \cite{Pie84}, \cite{Com97} (extrapolation of
simultaneous projection), \cite{LMWX12}, \cite{Ceg16}, \cite{CRZ20}
(extrapolation of a Landweber type operator) or CC12 (extrapolation of
cyclic projection).

In this and in the next section we answer the following questions:

\begin{enumerate}
\item Let $T$ and $U$ be SPCs with $\limfunc{Fix}(UT)\neq \emptyset $. What
should be supposed on the extrapolation function $\sigma :\mathcal{H}%
\rightarrow (0,+\infty )$ in order to receive the weak convergence of
sequences $x^{k}$ generated by the iteration $x^{k+1}=(UT)_{\sigma }(x^{k})$%
, $x^{0}\in \mathcal{H}$, to a fixed point of $UT$?

\item Let $T$ and $U$ be demicontractions with $\limfunc{Fix}T\cap \limfunc{%
Fix}U\neq \emptyset $. What should be supposed on the extrapolation function 
$\sigma :\mathcal{H}\rightarrow (0,+\infty )$ in order to receive the weak
convergence of sequences $x^{k}$ generated by the iteration $%
x^{k+1}=(UT)_{\sigma }(x^{k})$, $x^{0}\in \mathcal{H}$, to a common fixed
point of $U$ and $T$?
\end{enumerate}

Before we formulate our main result of this section, we prove the following
Lemma.

\begin{lemma}
\label{l-a+b}Let $\lambda ,\mu >0$ be such that $\lambda \mu <4$ and $a,b\in 
\mathcal{H}$ be such that $\Vert a\Vert ^{2}+\Vert b\Vert ^{2}>0$. Then 
\begin{equation}
\frac{1}{\lambda }\Vert a\Vert ^{2}+\frac{1}{\mu }\Vert b\Vert ^{2}+\langle
a,b\rangle >0  \label{e-1/lam}
\end{equation}%
and 
\begin{equation}
0<\frac{\Vert a+b\Vert ^{2}}{\frac{1}{\lambda }\Vert a\Vert ^{2}+\frac{1}{%
\mu }\Vert b\Vert ^{2}+\langle a,b\rangle }\leq \nu ^{\ast }=\frac{4(\lambda
+\mu -\lambda \mu )}{4-\lambda \mu }\text{.}  \label{e-a+b}
\end{equation}
\end{lemma}

\begin{proof}
Inequality (\ref{e-1/lam}) is clear if $\langle a,b\rangle \geq 0$. If $%
\langle a,b\rangle <0$ then it follows from the properties of the inner
product and from the assumption that 
\begin{equation*}
\frac{1}{\lambda }\Vert a\Vert ^{2}+\frac{1}{\mu }\Vert b\Vert ^{2}+\langle
a,b\rangle =\Vert \frac{1}{\sqrt{\lambda }}a+\frac{1}{\sqrt{\mu }}b\Vert
^{2}+(1-\frac{2}{\sqrt{\lambda \mu }})\langle a,b\rangle >0\text{.}
\end{equation*}%
Moreover, $\lambda +\mu -\lambda \mu >0$, because $\lambda \mu <4$ (see
Appendix). A direct calculus shows that 
\begin{eqnarray*}
&&4(\lambda +\mu -\lambda \mu )(\frac{1}{\lambda }\Vert a\Vert ^{2}+\frac{1}{%
\mu }\Vert b\Vert ^{2}+\langle a,b\rangle )-(4-\lambda \mu )\Vert a+b\Vert
^{2} \\
&=&\left\Vert \sqrt{\frac{\mu }{\lambda }}\left( \left\vert \lambda
-2\right\vert \right) a\pm \sqrt{\frac{\lambda }{\mu }}\left( \left\vert \mu
-2\right\vert \right) b\right\Vert ^{2}\geq 0\text{,}
\end{eqnarray*}%
where the sign $\pm $ should be replaced by $+$ if $(\lambda -2)(\mu -2)\leq
0$ and by $-$ if $(\lambda -2)(\mu -2)>0$. This proves inequality (\ref%
{e-a+b}) which completes the proof.
\end{proof}

\bigskip

Similarly as before, for operators $T$ and $U$ with $\limfunc{Fix}(UT)\neq
\emptyset $ and for $x,y\in \mathcal{H}$ denote $a_{1}:=T(x)-x$, $%
a_{2}:=T(y)-y$, $b_{1}:=UT(x)-T(x)$, $b_{2}:=UT(y)-T(y)$. Moreover, for $%
\lambda ,\mu >0$ with $\lambda \mu <4$ denote 
\begin{equation}
\tau ^{\ast }(x,y):=\left\{ 
\begin{array}{ll}
\frac{\Vert (a_{1}-a_{2})+(b_{1}-b_{2})\Vert ^{2}}{\frac{1}{\lambda }\Vert
a_{1}-a_{2}\Vert ^{2}+\frac{1}{\mu }\Vert b_{1}-b_{2}\Vert ^{2}+\langle
a_{1}-a_{2},b_{1}-b_{2}\rangle }\text{, } & \text{if }x,y\notin \limfunc{Fix}%
(UT)\text{,} \\ 
1\text{,} & \text{otherwise.}%
\end{array}%
\right.  \label{e-tau*}
\end{equation}%
By Lemma \ref{l-a+b} with $a:=a_{1}-a_{2}$ and $b:=b_{1}-b_{2}$ and by the
definition of of $\nu ^{\ast }$ given by (\ref{e-gam-sol}), the function $%
\tau ^{\ast }$ is well defined and $0<\tau ^{\ast }(x,y)\leq \nu ^{\ast }$.
Moreover, for $y=z\in \limfunc{Fix}(UT)$, 
\begin{equation*}
\tau ^{\ast }(x,z):=\left\{ 
\begin{array}{ll}
\frac{\Vert a_{1}+b_{1}\Vert ^{2}}{\frac{1}{\lambda }\Vert a_{1}-a_{2}\Vert
^{2}+\frac{1}{\mu }\Vert b_{1}-b_{2}\Vert ^{2}+\langle
a_{1}-a_{2},b_{1}-b_{2}\rangle }\text{, } & \text{if }x\notin \limfunc{Fix}%
(UT)\text{,} \\ 
1\text{,} & \text{otherwise,}%
\end{array}%
\right.
\end{equation*}%
because $a_{2}+b_{2}=UT(z)-z=0$.

\begin{theorem}
\label{t-ext-comp-RFNE}Let $\lambda ,\mu >0$ with $\lambda \mu <4$, $T$ be $%
\lambda $-RFNE and $U$ be $\mu $-RFNE with $\limfunc{Fix}(UT)\neq \emptyset $%
. Let $x\in \mathcal{H}$ and $z\in \limfunc{Fix}(UT)$ be fixed and the
function $\tau :\mathcal{H}\rightarrow (0,+\infty )$ be such that 
\begin{equation}
\tau ^{\ast }(x,z)\leq \tau (x)\leq \nu ^{\ast }  \label{e-tau}
\end{equation}%
for $x\notin \limfunc{Fix}(UT)$. Then the operator $UT$ is a generalized $%
\tau $-relaxed cutter, consequently, $(UT)_{1/\tau }$ is a cutter which is
an extrapolation of the cutter $(UT)_{1/\nu ^{\ast }}$. Moreover, $%
(UT)_{1/\tau }$ satisfies the demi closedness principle.
\end{theorem}

\begin{proof}
Let $x\in \mathcal{H}$. By Lemma \ref{l-a+b} with $a:=a_{1}-a_{2}$ and $%
b:=b_{1}-b_{2}$, there is a function $\tau $ satisfying (\ref{e-tau}). Note
that for $y=z\in \limfunc{Fix}(UT)$ it holds $a_{2}+b_{2}=UT(z)-z=0$. We
prove that $UT$ is a generalized $\tau (x)$-relaxed cutter, i.e., 
\begin{equation*}
\tau (x)\langle z-x,UT(x)-x\rangle \geq \Vert UT(x)-x\Vert ^{2}\text{.}
\end{equation*}%
For $z\in \limfunc{Fix}(UT)$ we have 
\begin{eqnarray*}
&&\langle z-x,UT(x)-x\rangle -\frac{1}{\tau (x)}\Vert UT(x)-x\Vert ^{2} \\
&=&\langle z-x,(UT(x)-x)-(UT(z)-z)\rangle -\frac{1}{\tau (x)}\Vert
UT(x)-x-(UT(z)-z)\Vert ^{2}\text{.}
\end{eqnarray*}%
Thus, Theorem \ref{t-comp-RFNE}(i) with $y=z$ and $\nu =\tau (x)>0$ yields
that it is enough to prove that 
\begin{equation}
(\frac{1}{\lambda }-\frac{1}{\tau (x)})\Vert a_{1}-a_{2}\Vert ^{2}+(\frac{1}{%
\mu }-\frac{1}{\tau (x)})\Vert b_{1}-b_{2}\Vert ^{2}+(1-\frac{2}{\tau (x)}%
)\langle a_{1}-a_{2},b_{1}-b_{2}\rangle \geq 0\text{.}  \label{e-ni}
\end{equation}%
If $x\notin \limfunc{Fix}(UT)$ then (\ref{e-ni}) is equivalent to 
\begin{equation*}
\tau (x)\geq \frac{\Vert a_{1}+b_{1}\Vert ^{2}}{\frac{1}{\lambda }\Vert
a_{1}-a_{2}\Vert ^{2}+\frac{1}{\mu }\Vert b_{1}-b_{2}\Vert ^{2}+\langle
a_{1}-a_{2},b_{1}-b_{2}\rangle }=\tau ^{\ast }(x,z)\text{.}
\end{equation*}%
Consequently, $UT$ is a generalized $\tau $-relaxed cutter and $(UT)_{1/\tau
}$ is a cutter. This and the second inequality in (\ref{e-tau}) yield that $%
(UT)_{1/\tau }$ is an extrapolation of $(UT)_{1/\nu ^{\ast }}$. By Theorem %
\ref{t-comp-RFNE}(ii), $UT$ is $\nu ^{\ast }$-RFNE, thus $(UT)_{1/\nu ^{\ast
}}$ is a FNE. Because $\limfunc{Fix}(UT)\neq \emptyset $, this yields that $%
(UT)_{1/\nu ^{\ast }}$ is a cutter. Moreover, $(UT)_{1/\nu ^{\ast }}$
satisfies the demi closedness principle as an NE operator, thus $%
(UT)_{1/\tau }$ also satisfied the demi closedness principle as an
extrapolation of $(UT)_{1/\nu ^{\ast }}$.
\end{proof}

\bigskip

In order to apply Theorem \ref{t-ext-comp-RFNE} we should be able to
evaluate $\tau ^{\ast }(x,z)$ for some $z\in \limfunc{Fix}UT$. However, in
general, this is a hard task because we do not know $z\in \limfunc{Fix}(UT)$
explicitly. Nevertheless, is some cases, one can define a function $\tau (x)$
satisfying (\ref{e-tau}) without knowledge of $z\in \limfunc{Fix}(UT)$. This
function can be applied for a construction of a cutter which is an
extrapolation of the cutter $(UT)_{1/\nu ^{\ast }}$.

\begin{example}
%TCIMACRO{\TeXButton{rm}{\rm\ }}%
%BeginExpansion
\rm\ %
%EndExpansion
Suppose that $\lambda \in \lbrack 1,4)$, $\mu =1$. Then, obviously, $\lambda
\mu <4$. Further, let $T:=(P_{A})_{\lambda }$, where $A\subseteq \mathcal{H}$
is closed convex, $U:=P_{B}$, where $B$ is a closed affine subspace and $%
\limfunc{Fix}(P_{B}P_{A})\neq \emptyset $. It is easily seen that $\limfunc{%
Fix}(UT)=\limfunc{Fix}(P_{B}P_{A})$. Let $z\in \limfunc{Fix}(P_{B}P_{A})$ be
arbitrary, $w:=Tz$ and $d:=z-P_{A}(z)$. Then $\Vert d\Vert $ is the distance
between $A$ and $B$. We prove that:

\begin{enumerate}
\item[$\mathrm{(i)}$] For $x\in B$ it holds that 
\begin{equation}
\tau ^{\ast }(x,z)=\left\{ 
\begin{array}{ll}
\frac{\Vert a_{1}+b_{1}\Vert ^{2}}{\frac{1}{\lambda }\Vert a_{1}\Vert
^{2}+\Vert b_{1}\Vert ^{2}+\langle a_{1},b_{1}\rangle +\lambda \Vert d\Vert
^{2}-2\langle b_{1},d\rangle }\text{, } & \text{if }x\notin \limfunc{Fix}(UT)%
\text{,} \\ 
1\text{,} & \text{otherwise.}%
\end{array}%
\right.  \label{e-tau1*}
\end{equation}

\item[$\mathrm{(ii)}$] For 
\begin{equation}
\bar{\tau}(x):=\left\{ 
\begin{array}{ll}
\frac{\Vert a_{1}+b_{1}\Vert ^{2}}{\frac{1}{\lambda }\Vert a_{1}\Vert
^{2}+\Vert b_{1}\Vert ^{2}+\langle a_{1},b_{1}\rangle -\frac{1}{\lambda }%
\Vert b_{1}\Vert ^{2}}\text{,} & \text{if }x\notin \limfunc{Fix}(UT)\text{,}
\\ 
1\text{,} & \text{otherwise,}%
\end{array}%
\right.  \label{e-tau-}
\end{equation}%
where $x\in B$, it holds $\bar{\tau}(x)\geq \tau ^{\ast }(x,z)$ and the
operator $UT$ is a generalized $\bar{\tau}$-relaxed cutter, consequently, $%
(UT)_{1/\bar{\tau}}$ is a cutter.

\item[$\mathrm{(iii)}$] For $\hat{\tau}:=\min \{\bar{\tau},\nu ^{\ast }\}$,
where $\nu ^{\ast }=\nu ^{\ast }(\lambda ,1)$ the operator $UT$ is a
generalized $\hat{\tau}$-relaxed cutter and $(UT)_{1/\hat{\tau}}$ is a
cutter which is an extrapolation of the cutter $(UT)_{1/\nu ^{\ast }}$.
Moreover, $(UT)_{1/\hat{\tau}}$ satisfies the demi closedness principle.
\end{enumerate}
\end{example}

\begin{proof}
Let $x\in B$. For $y=z\in \limfunc{Fix}(UT)$ we have $b_{2}=-a_{2}$.
Moreover, 
\begin{equation*}
T(z)-z=a_{2}=(P_{A})_{\lambda }(z)-z=\lambda (P_{A}(z)-z)=-\lambda d\text{.}
\end{equation*}%
Let $x\in B\setminus \limfunc{Fix}(UT)$. We have $UT(x)=P_{B}T(x)\in B$ and $%
z=P_{B}(w)$. These facts, the properties of the metric projection and the
affinity of $B$ yield 
\begin{equation*}
\langle a_{1}+b_{1},d\rangle =\langle a_{1}+b_{1},z-P_{A}(z)\rangle =\frac{1%
}{\lambda }\langle P_{B}T(x)-x,P_{B}T(z)-T(z)\rangle =0\text{,}
\end{equation*}%
and 
\begin{equation*}
\langle a_{1}+b_{1},b_{1}\rangle =\langle P_{B}T(x)-x,P_{B}T(x)-T(x)\rangle
=0
\end{equation*}%
i.e. $\langle a_{1},d\rangle =-\langle b_{1},d\rangle $ and $\langle
a_{1},b_{1}\rangle =-\Vert b_{1}\Vert ^{2}$. Moreover, $\lambda \Vert d\Vert
^{2}-2\Vert b_{1}\Vert \cdot \Vert d\Vert \geq -\frac{1}{\lambda }\Vert
b_{1}\Vert ^{2}$, because the function $f(\xi ):=\lambda \xi ^{2}-2\beta \xi 
$ attains its minimum at $\xi =\beta /\lambda $ equal to $-\beta
^{2}/\lambda $. These facts yield 
\begin{eqnarray}
&&\frac{1}{\lambda }\Vert a_{1}-a_{2}\Vert ^{2}+\Vert b_{1}-b_{2}\Vert
^{2}+\langle a_{1}-a_{2},b_{1}-b_{2}\rangle  \label{A} \\
&=&\frac{1}{\lambda }\Vert a_{1}+\lambda d\Vert ^{2}+\Vert b_{1}-\lambda
d\Vert ^{2}+\langle a_{1}+\lambda d,b_{1}-\lambda d\rangle  \label{B} \\
&=&\frac{1}{\lambda }\Vert a_{1}\Vert ^{2}+\Vert b_{1}\Vert ^{2}+\langle
a_{1},b_{1}\rangle +\lambda \Vert d\Vert ^{2}-2\langle b_{1},d\rangle
\label{C} \\
&\geq &\frac{1}{\lambda }\Vert a_{1}\Vert ^{2}+\Vert b_{1}\Vert ^{2}+\langle
a_{1},b_{1}\rangle +\lambda \Vert d\Vert ^{2}-2\Vert b_{1}\Vert \cdot \Vert
d\Vert  \label{D} \\
&\geq &\frac{1}{\lambda }\Vert a_{1}\Vert ^{2}+\Vert b_{1}\Vert ^{2}+\langle
a_{1},b_{1}\rangle -\frac{1}{\lambda }\Vert b_{1}\Vert ^{2}\text{.}
\label{E}
\end{eqnarray}%
Part (i) follows from equalities (\ref{A})-(\ref{C}) and from (\ref{e-tau*}%
). Note that $\Vert a_{1}\Vert >\Vert b_{1}\Vert $, because $x\in B\setminus 
\limfunc{Fix}(UT)$, consequently, $\frac{1}{\lambda }\Vert a_{1}\Vert
^{2}+\Vert b_{1}\Vert ^{2}+\langle a_{1},b_{1}\rangle -\frac{1}{\lambda }%
\Vert b_{1}\Vert ^{2}>0$. Thus, $\bar{\tau}$ is well defined and
inequalities (\ref{D})-(\ref{E}) show that $\bar{\tau}(x)\geq \tau ^{\ast
}(x,z)$. This shows that $UT$ is a generalized $\bar{\tau}$-relaxed cutter,
consequently, $(UT)_{1/\bar{\tau}}$ is a cutter. This proves part (ii).

Unfortunately, $\bar{\tau}(x)$ needs not to satisfy $\bar{\tau}(x)\leq \nu
^{\ast }(\lambda ,1)$, thus, $(UT)_{1/\bar{\tau}}$ needs not to be an
extrapolation of $(UT)_{1/\nu ^{\ast }}$. Thus, we introduce the function $%
\hat{\tau}:=\min \{\bar{\tau},\nu ^{\ast }\}$, where $\nu ^{\ast }=\nu
^{\ast }(\lambda ,1)$, which obviously satisfies $\tau ^{\ast }(x,z)\leq 
\hat{\tau}(x)\leq \nu ^{\ast }(\lambda ,1)$. Similarly as before, this
yields, that $(UT)_{1/\hat{\tau}}$ is a cutter which is an extrapolation of
the cutter $(UT)_{1/\nu ^{\ast }}$ and satisfies the demi closedness
principle.
\end{proof}

\bigskip

If $\limfunc{Fix}T\cap \limfunc{Fix}U\neq \emptyset $ then the function $%
\tau ^{\ast }$ can be evaluated without knowledge of $z\in \limfunc{Fix}%
T\cap \limfunc{Fix}U$, because for $y=z\in \limfunc{Fix}T\cap \limfunc{Fix}U$
we have $a_{2}=b_{2}=0$. Consequently, for $x\notin \limfunc{Fix}T\cap 
\limfunc{Fix}U$ and for arbitrary $z\in \limfunc{Fix}T\cap \limfunc{Fix}U$
we have 
\begin{equation}
\tau ^{\ast }(x):=\tau ^{\ast }(x,z)=\left\{ 
\begin{array}{ll}
\frac{\Vert a_{1}+b_{1}\Vert ^{2}}{\frac{1}{\lambda }\Vert a_{1}\Vert ^{2}+%
\frac{1}{\mu }\Vert b_{1}\Vert ^{2}+\langle a_{1},b_{1}\rangle }\text{, } & 
\text{if }x\notin \limfunc{Fix}T\cap \limfunc{Fix}U\text{,} \\ 
1\text{,} & \text{otherwise.}%
\end{array}%
\right.  \label{e-tau2*}
\end{equation}%
In this case it enough to suppose that $T$ and $U$ are relaxed cutters
satisfying the demi closedness principle instead of being RFNE operators.

\begin{corollary}
\label{c-ext-comp-RC}Let $\lambda ,\mu >0$ with $\lambda \mu <4$, $T$ be a $%
\lambda $-relaxed cutter and $U$ be a $\mu $-relaxed cutter with $\limfunc{%
Fix}T\cap \limfunc{Fix}U\neq \emptyset $. Let $x\in \mathcal{H}$ be
arbitrary and the function $\tau :\mathcal{H}\rightarrow (0,+\infty )$ be
such that 
\begin{equation}
\tau ^{\ast }(x)\leq \tau (x)\leq \nu ^{\ast }  \label{e-tau*(x)}
\end{equation}%
for $x\notin \limfunc{Fix}T\cap \limfunc{Fix}U$, where $\tau ^{\ast }(x)$ is
given by $\mathrm{(\ref{e-tau2*})}$. Then the operator $UT$ is a generalized 
$\tau $-relaxed cutter, consequently, $(UT)_{1/\tau }$ is a cutter which is
an extrapolation of the cutter $(UT)_{1/\nu ^{\ast }}$. Moreover, if $T$ and 
$U$ satisfy the demi closedness principle, then $(UT)_{1/\tau }$ also
satisfies the demi closedness principle.
\end{corollary}

\section{\label{s-4}Convergence properties of algorithms employing strict
pseudocontractions and demicontractions}

Let $T,U:\mathcal{H}\rightarrow \mathcal{H}$ be two relaxed cutters
(equivalently, $T,U$ are two demicontractions) with $\limfunc{Fix}(UT)\neq
\emptyset $. For a relaxation function $\sigma :\mathcal{H}\rightarrow
(0,+\infty )$ we define 
\begin{equation*}
V:=(UT)_{\sigma }\text{,}
\end{equation*}%
a $\sigma $-generalized relaxation of $UT$. We consider the iteration 
\begin{equation}
x^{k+1}=V_{\lambda _{k}}(x^{k})=(UT)_{\lambda _{k}\sigma }(x^{k})\text{,}
\label{e-iterV-lamk}
\end{equation}%
where $x^{0}\in \mathcal{H}$ is arbitrary and the relaxation parameter $%
\lambda _{k}\in \lbrack \varepsilon ,2-\varepsilon ]$ for some small $%
\varepsilon >0$. In this section we give conditions under which sequences
generated by iteration (\ref{e-iterV-lamk}) converge weakly to an element of 
$\limfunc{Fix}(UT)\neq \emptyset $. We use the same notation as in
Subsection \ref{ss-3.3}.

\begin{theorem}
\label{t-iter-RFNE}Let $\lambda ,\mu >0$ be such that $\lambda \mu <4$, $T:%
\mathcal{H}\rightarrow \mathcal{H}$ be $\lambda $-RFNE and $U:\mathcal{H}%
\rightarrow \mathcal{H}$ be $\mu $-RFNE with $\limfunc{Fix}(UT)\neq
\emptyset $. Then the sequence $\{x^{k}\}_{k=0}^{\infty }$ generated by
iteration $\mathrm{(\ref{e-iterV-lamk})}$, where $\sigma =1/\tau $ with $%
\tau $ satisfying $\mathrm{(\ref{e-tau})}$ for all $x\in \mathcal{H}$ and
for some $z\in \limfunc{Fix}(UT)$ converges weakly to an element of $%
\limfunc{Fix}(UT)$.
\end{theorem}

\begin{proof}
The theorem follows directly from Theorem \ref{t-ext-comp-RFNE} and from
Proposition \ref{p-WC}.
\end{proof}

\bigskip

Setting $\tau (x)=\nu ^{\ast }$ for $x\notin \limfunc{Fix}(UT)$ in Theorem %
\ref{t-iter-RFNE} we receive the following result.

\begin{corollary}
\label{c-iter-SPCa}Let $\lambda ,\mu >0$ be such that $\lambda \mu <4$, $T:%
\mathcal{H}\rightarrow \mathcal{H}$ be $\lambda $-RFNE and $U:\mathcal{H}%
\rightarrow \mathcal{H}$ be $\mu $-RFNE with $\limfunc{Fix}(UT)\neq
\emptyset $. Then the sequence $\{x^{k}\}_{k=0}^{\infty }$ generated by the
iteration 
\begin{equation}
x^{k+1}=(UT)_{\lambda _{k}/\nu ^{\ast }}(x^{k})\text{,}  \label{e-RASPC}
\end{equation}%
where $x^{0}\in \mathcal{H}$, $\lambda _{k}\in \lbrack \varepsilon
,2-\varepsilon ]$ for some small $\varepsilon >0$ and $\nu ^{\ast }$ is
defined by $\mathrm{(\ref{e-gam-sol}),}$ converges weakly to an element of $%
\limfunc{Fix}(UT)$.
\end{corollary}

We call the iteration (\ref{e-RASPC}) a \textit{relaxed alternating strict
pseudocontraction method} (RASPC).

\begin{remark}
%TCIMACRO{\TeXButton{rm}{\rm\ }}%
%BeginExpansion
\rm\ %
%EndExpansion
If $\lambda =\mu =2$, $T:=(P_{A})_{\lambda }$ and $U:=(P_{B})_{\mu }$, then $%
V:=(UT)_{\rho }$ with $\rho =\frac{1}{2}$ is, actually, the
Douglas--Rachford operator which is FNE. Thus, Theorem \ref{t-iter-RFNE}
extends the convergence of the Douglas--Rachford method to the case $\lambda
\mu <4$.
\end{remark}

Applying Proposition \ref{c-SPC-RFNE}, Theorem \ref{t-iter-RFNE} can be
equivalently formulated as follows.

\begin{corollary}
\label{c-iter-SPC}Let $\alpha ,\beta \in (-\infty ,1)$ be such that $\alpha
+\beta <\alpha \beta $, $T:\mathcal{H}\rightarrow \mathcal{H}$ be an $\alpha 
$-SPC and $U:\mathcal{H}\rightarrow \mathcal{H}$ be a $\beta $-SPC with $%
\limfunc{Fix}(UT)\neq \emptyset $. Then the sequence $\{x^{k}\}_{k=0}^{%
\infty }$ generated by iteration $\mathrm{(\ref{e-iterV-lamk})}$, where the
generalized relaxation function $\sigma =1/\tau $ with $\tau $ satisfying 
\begin{equation}
\frac{2\Vert a_{1}+b_{1}\Vert ^{2}}{(1-\alpha )\Vert a_{1}-a_{2}\Vert
^{2}+(1-\beta )\Vert b_{1}-b_{2}\Vert ^{2}+2\langle
a_{1}-a_{2},b_{1}-b_{2}\rangle }\leq \tau (x)\leq \frac{2(\alpha +\beta )}{%
\alpha +\beta -\alpha \beta }  \label{e-tau2}
\end{equation}%
for arbitrary $x\in \mathcal{H\setminus }\limfunc{Fix}(UT)$ and some $y=z\in 
\limfunc{Fix}(UT)$, converges weakly to an element of $\limfunc{Fix}UT$.
\end{corollary}

\begin{theorem}
\label{t-iter-UV}Let $\lambda ,\mu >0$, $T:\mathcal{H}\rightarrow \mathcal{H}
$ be a $\lambda $-relaxed cutter and $U:\mathcal{H}\rightarrow \mathcal{H}$
be a $\mu $-relaxed cutter with $\limfunc{Fix}T\cap \limfunc{Fix}U\neq
\emptyset $. Suppose that $T$ and $U$ satisfy the demi closedness principle.
If $\lambda \mu <4$ then the sequence $\{x^{k}\}_{k=0}^{\infty }$ generated
by iteration $\mathrm{(\ref{e-iterV-lamk})}$, where $\sigma =1/\tau $ with $%
\tau $ satisfying $\mathrm{(\ref{e-tau*(x)})}$ and $\tau ^{\ast }(x)$
defined by $\mathrm{(\ref{e-tau2*})}$ for all $x\in \mathcal{H}$, converges
weakly to an element of $\limfunc{Fix}T\cap \limfunc{Fix}U$.
\end{theorem}

\begin{proof}
Suppose that $\lambda \mu <4$. By Corollary \ref{c-ext-comp-RC}, $%
(UT)_{\sigma }$ is a cutter which is an extrapolation of the cutter $%
(UT)_{1/\nu ^{\ast }}$ and $(UT)_{\sigma }$ satisfies the demi closedness
principle. The remaining part follows from Proposition \ref{p-WC}.
\end{proof}

\bigskip

We call iteration $\mathrm{(\ref{e-iterV-lamk})}$, where $\sigma =1/\tau $
with $\tau $ satisfying $\mathrm{(\ref{e-tau*(x)})}$ and $\tau ^{\ast }(x)$
defined by $\mathrm{(\ref{e-tau2*})}$ an \textit{extrapolated alternating
demicontraction} (EADC) method.

Theorem \ref{t-iter-UV} can be equivalently formulated as follows.

\begin{corollary}
\label{c-iter-UV}Let $\alpha ,\beta \in (-\infty ,1)$, $T:\mathcal{H}%
\rightarrow \mathcal{H}$ be an $\alpha $-demicontraction and $U:\mathcal{H}%
\rightarrow \mathcal{H}$ be a $\beta $-demicontraction with $\limfunc{Fix}%
T\cap \limfunc{Fix}U\neq \emptyset $. Suppose that $T$ and $U$ satisfy the
demi closedness principle. If $\alpha +\beta <\alpha \beta $ then the
sequence $\{x^{k}\}_{k=0}^{\infty }$ generated by the iteration $\mathrm{(%
\ref{e-iterV-lamk}),}$ where $\sigma =1/\tau $ with $\tau $ satisfying 
\begin{equation*}
\frac{2\Vert a_{1}+b_{1}\Vert ^{2}}{(1-\alpha )\Vert a_{1}\Vert
^{2}+(1-\beta )\Vert b_{1}\Vert ^{2}+2\langle a_{1},b_{1}\rangle }\leq \tau
(x)\leq \frac{2(\alpha +\beta )}{\alpha +\beta -\alpha \beta }
\end{equation*}%
for arbitrary $x\in \mathcal{H\setminus }\limfunc{Fix}(UT)$ converges weakly
to an element of $\limfunc{Fix}T\cap \limfunc{Fix}U$.
\end{corollary}

Contrary to Theorem \ref{t-iter-UV} and Corollary \ref{c-iter-UV}, in
Theorem \ref{t-iter-RFNE} and in Corollary \ref{c-iter-SPC} we do not
suppose that $\limfunc{Fix}T\cap \limfunc{Fix}U\neq \emptyset $. We only
suppose that $\limfunc{Fix}(UT)\neq \emptyset $. In the case of the
Douglas--Rachford operator, i.e. $T=2P_{A}-\limfunc{Id}$ and $U=2P_{B}-%
\limfunc{Id}$ for closed convex subsets $A,B\subseteq \mathcal{H}$ it is
well known that $\limfunc{Fix}(UT)\neq \emptyset $ if and only if $A\cap
B\neq \emptyset $ \cite[Prop. 7]{ACT20}. The first part of the proposition
below extends \cite[Prop. 7]{ACT20}, where the case $\lambda =\mu =2$ was
proved, and is a special case of \cite[Lemma 4.1(iii)]{DP19}. The second
part shows that for $\lambda +\mu \neq \lambda \mu $ the nonemptiness of $%
\limfunc{Fix}(UT)$ is also possible if $A\cap B=\emptyset $.

\begin{proposition}
Let $\lambda ,\mu >0$, $A,B\subseteq \mathcal{H}$ be nonempty, closed and
convex, $T:=(P_{A})_{\lambda }$, $U:=(P_{B})_{\mu }$ and $V:=UT$.

\begin{enumerate}
\item[$\mathrm{(i)}$] If $\lambda +\mu =\lambda \mu $ then%
\begin{equation}
\limfunc{Fix}V\neq \emptyset \Longleftrightarrow A\cap B\neq \emptyset \text{%
;}  \label{e-FixV}
\end{equation}%
If, moreover, $A\cap B\neq \emptyset $ then 
\begin{equation}
P_{A}(\limfunc{Fix}V)=A\cap B\text{.}  \label{e-PA(FixV)}
\end{equation}

\item[$\mathrm{(ii)}$] If $\lambda +\mu \neq \lambda \mu $ then there are $%
A,B$ with $A\cap B=\emptyset $ and $\limfunc{Fix}V\neq \emptyset $.
\end{enumerate}
\end{proposition}

\begin{proof}
(i) Suppose that $\lambda +\mu =\lambda \mu $. Clearly, $A\cap B\subseteq 
\limfunc{Fix}V$, so, if $A\cap B\neq \emptyset $ then $\limfunc{Fix}V\neq
\emptyset $. Suppose that $\limfunc{Fix}V\neq \emptyset $. By $V=(\mu
P_{B}+(1-\mu )\limfunc{Id})(P_{A})_{\lambda }$ and by $\lambda +\mu =\lambda
\mu $, 
\begin{equation}
z\in \limfunc{Fix}V\Longleftrightarrow P_{B}(P_{A})_{\lambda }(z)=P_{A}(z)%
\text{.}  \label{e-FixV-PB}
\end{equation}%
This yields $A\cap B\neq \emptyset $ which proves (\ref{e-FixV}). Now we
prove (\ref{e-PA(FixV)}). If $x\in A\cap B$ then $x\in \limfunc{Fix}V$ and $%
P_{A}(x)=x$, thus $x\in P_{A}(\limfunc{Fix}V)$. Let now $x\in P_{A}(\limfunc{%
Fix}V)$. Then $x\in A$ and there is $z\in \limfunc{Fix}V$ such that $%
x=P_{A}(z)$. By (\ref{e-FixV-PB}), $P_{B}(P_{A})_{\lambda }(z)=x$, thus $%
x\in B$. This proves (\ref{e-PA(FixV)}).

(ii) Suppose that $\lambda +\mu \neq \lambda \mu $. Let $A,B\subseteq 
\mathcal{H}$ be two disjoint hyperplanes. Let $a\in A$ and $b:=P_{B}(a)$. If
we set 
\begin{equation*}
x=\frac{\lambda a+\mu b-\lambda \mu a}{\lambda +\mu -\lambda \mu }
\end{equation*}%
then we obtain $V(x)=x$.
\end{proof}

\begin{example}
%TCIMACRO{\TeXButton{rm}{\rm\ }}%
%BeginExpansion
\rm\ %
%EndExpansion
Let $A,B\subseteq \mathcal{H}$ be nonempty closed convex subsets that share
a common point, $T=(P_{A})_{\lambda }$ and $U=(P_{B})_{\mu }$, where $%
\lambda ,\mu >0$ and $\lambda +\mu =4$. In particular, if $\lambda =\mu =2$
then $UT$ is NE and $(UT)_{\frac{1}{2}}$ is, actually, the Douglas--Rachford
operator. Now suppose that $\lambda ,\mu \neq 2$. We have $\lambda \mu <4$
and $\nu ^{\ast }=\nu (\lambda ,\mu )=4$. By Theorem \ref{t-comp-RFNE}%
(ii)-(iii) the operator $UT$ is $4$-RFNE and $\limfunc{Fix}UT=A\cap B$,
consequently, the operator $V$ defined by 
\begin{equation*}
V(x):=(UT)_{\frac{1}{4}}(x)=x+\frac{1}{4}(UT(x)-x)
\end{equation*}%
is FNE with $\limfunc{Fix}V=A\cap B$ and the sequence $\{x^{k}\}_{k=0}^{%
\infty }$ generated by the iteration 
\begin{equation}
x^{k+1}=x+\frac{\sigma _{k}}{4}(UT(x^{k})-x^{k}),  \label{e-gRD}
\end{equation}%
where $x^{0}\in \mathcal{H}$ is arbitrary and $\sigma _{k}\in \lbrack
\varepsilon ,2-\varepsilon ]$ for some $\varepsilon \in (0,1)$ converges
weakly to some $z\in \limfunc{Fix}V$. By Corollary \ref{c-iter-UV}, the
convergence also holds if $T$ and $U$ are demicontractions (equivalently,
relaxed cutters) with $\limfunc{Fix}T=A$ and $\limfunc{Fix}U=B$ which
satisfy the demi closedness principle. By Theorem \ref{t-iter-RFNE}, the
convergence also holds if we suppose that $T$ and $U$ are RFNE (or,
equivalently, SPC) with $\limfunc{Fix}(UT)\neq \emptyset $ (even if $A\cap
B=\emptyset $). On Fig. 2 we compare the behavior of the DR iteration and
iteration (\ref{e-gRD}) with $\lambda =3$, $\mu =1$ and $\sigma _{k}=1$ (a
relaxed alternating strict pseudocontraction method), where $A,B\subseteq 
\mathcal{H}$ are two intersecting hyperplanes. \vspace{-0.5cm}%
\begin{equation*}
\FRAME{itbpFU}{2.7766in}{2.1918in}{0in}{\Qcb{Fig. 2. Behavior of DR method
and RASPC metod}}{}{test-linear.emf}{\special{language "Scientific
Word";type "GRAPHIC";display "USEDEF";valid_file "F";width 2.7766in;height
2.1918in;depth 0in;original-width 10.3704in;original-height
7.7782in;cropleft "0";croptop "1";cropright "1";cropbottom "0";filename
'test-linear.emf';file-properties "XNPEU";}}
\end{equation*}%
\vspace{-0.5cm}
\end{example}

\begin{example}
%TCIMACRO{\TeXButton{rm}{\rm\ }}%
%BeginExpansion
\rm\ %
%EndExpansion
Let $A\subseteq \mathcal{H}$ be a ball and $B$ be a hyperplane tangent to $A$%
, $T:=(P_{A})_{\lambda },U:=(P_{B})_{\mu }$. On Fig. 3 we compare the
behavior of the DR method with an extrapolated alternating demicontraction
method $x^{k+1}=(UT)_{1/\tau ^{\ast }}(x^{k})$ for $\lambda =3$ and $\mu =1,$
where $\tau ^{\ast }$ is given by (\ref{e-tau2*}).\vspace{-0.5cm}%
\begin{equation*}
\FRAME{itbpFU}{2.5045in}{2.0082in}{0in}{\Qcb{Fig. 3. Behavior of DR method
and EADC method}}{}{t.emf}{\special{language "Scientific Word";type
"GRAPHIC";display "USEDEF";valid_file "F";width 2.5045in;height
2.0082in;depth 0in;original-width 10.3704in;original-height
7.7782in;cropleft "0";croptop "1";cropright "1";cropbottom "0";filename
't.emf';file-properties "XNPEU";}}
\end{equation*}%
\vspace{-0.5cm}
\end{example}

\begin{example}
%TCIMACRO{\TeXButton{rm}{\rm\ }}%
%BeginExpansion
\rm\ %
%EndExpansion
Let us come back to Example \ref{ex-CQ} and to iteration (\ref{e-iter-Mou10}%
), where $\alpha ,\beta \in (-\infty ,1)$ and $\lambda >0$. We suppose for
simplicity, that $\mu _{k}$ is constant, $\mu _{k}=\mu >0$. Because $T$ is
an $\alpha $-demicontraction with $\limfunc{Fix}T=A^{-1}(\limfunc{Fix}%
S)=A^{-1}(Q)$ and $U$ is a $\beta $-demicontraction with $\limfunc{Fix}U=C$,
Corollary \ref{c-dc-dc} yields that $T_{\lambda }$ is a $\gamma $%
-demicontraction with $\gamma =1-\frac{1-\alpha }{\lambda }$ and $U_{\mu }$
is a $\delta $-demicontraction with $\delta =1-\frac{1-\beta }{\mu }$. Note
that, contrary to Example \ref{ex-CQ}, $T_{\lambda }$ is not SQNE if $%
\lambda >1-\alpha $ and $U_{\mu }$ is not SQNE if $\mu >1-\beta $. If $%
\gamma +\delta <\gamma \delta $ then, by Corollary \ref{c-comp-dc}, $U_{\mu
}T_{\lambda }$ is a $\nu $-demicontraction with $\nu =\frac{\gamma \delta }{%
\gamma +\delta }$. Now Corollary \ref{c-iter-UV} yields that the operator $%
V:=(U_{\mu }T_{\lambda })_{\tau }$ with $\tau :=\frac{2(\gamma +\delta )}{%
\gamma +\delta -\gamma \delta }$is a cutter and the sequence generated by
the iteration $x^{k+1}=V_{\sigma _{k}}(x^{k})$, where $x^{0}\in \mathcal{H}$
is arbitrary and $\sigma _{k}\in \lbrack \varepsilon ,2-\varepsilon ]$ for
some $\varepsilon \in (0,1)$, converges weakly to an element of $F:=C\cap
A^{-1}(Q)$. Note that Moudafi supposed in \cite{Mou10} that $\lambda \in
(0,1-\alpha )$, $\mu \in (0,1-\beta )$, $\tau =1$ and $\sigma _{k}=1$.
\end{example}

\section{Appendix}

\begin{lemma}
\label{l-A}Let $\lambda ,\mu >0$. If $\lambda \mu <4$ then $\lambda +\mu
>\lambda \mu $.
\end{lemma}

\begin{proof}
Let $\lambda \mu <4$. If $\lambda \leq 1$ then 
\begin{equation*}
\lambda +\mu (1-\lambda )\geq \lambda >0
\end{equation*}%
which yields $\lambda +\mu >\lambda \mu $. Let now $\lambda >1$ and suppose
that $\lambda +\mu \leq \lambda \mu $. Then we obtain%
\begin{equation*}
0\geq \lambda +\mu (1-\lambda )>\lambda +\frac{4}{\lambda }(1-\lambda )=%
\frac{(\lambda -2)^{2}}{\lambda }\geq 0\text{,}
\end{equation*}%
a contradiction.
\end{proof}

\begin{lemma}
\label{l-B}Let $\alpha ,\beta \in (-\infty ,1)$.

\begin{enumerate}
\item[$\mathrm{(i)}$] If $\alpha +\beta <\alpha \beta $ then $\alpha +\beta
<0$ and $\frac{\alpha \beta }{\alpha +\beta }<1$.

\item[$\mathrm{(ii)}$] If $\frac{\alpha \beta }{\alpha +\beta }<1$ and at
most one of $\alpha ,\beta \geq 0$ then $\alpha +\beta <\alpha \beta $.
\end{enumerate}
\end{lemma}

\begin{proof}
(i) Define $f(x)=1-\frac{2}{x}$ for $x>0$. Then (i) follows from Lemma \ref%
{l-A} by setting $\alpha =f(\lambda )$ and $\beta =f(\mu )$.

(ii) Let $\frac{\alpha \beta }{\alpha +\beta }<1$. If $\alpha ,\beta <0$
then $\alpha +\beta <0$, consequently, $\alpha \beta >\alpha +\beta $. By
the symmetry it is enough to consider the case $\alpha <0$, $\beta \geq 0$.
If $\beta =0$ then $\alpha +\beta =\alpha <0=\alpha \beta $. If $\beta >0$
then $\frac{\alpha +\beta }{\alpha \beta }>1$ which yields $\alpha +\beta
<\alpha \beta $.
\end{proof}

\end{document}